\documentclass[reqno]{amsart}
\usepackage[utf8]{inputenc}
\usepackage[T1]{fontenc}
\usepackage{enumerate}
\usepackage{lmodern, microtype, amssymb,comment}
\usepackage{algpseudocode}
\usepackage{algorithm}

\usepackage[pdftex,dvipsnames]{xcolor} % Colors
\usepackage{graphicx}
\usepackage{color,psfrag,caption}
\allowdisplaybreaks
\definecolor{darkgreen}{rgb}{0,0.6,0}
\definecolor{darkred}{rgb}{0.7,0,0}
\definecolor{darkblue}{rgb}{0,.2,.7}
 \usepackage[%
   colorlinks,
  citecolor=darkgreen, linkcolor=darkblue,
   pdfpagelabels=true,
   unicode=true,
   pdfusetitle
 ]{hyperref}
 
  \hypersetup{
 pdfauthor={Marcelo Actis, Pedro Morin, and Cornelia Schneider},
 pdftitle={AFEM},
 pdflang={en}
}

\usepackage[shadow,colorinlistoftodos,prependcaption,spanish,textsize=footnotesize]{todonotes}
\usepackage{xargs}   % Use more than one optional parameter in a new commands

\newcommandx{\todomarce}[2][1=]{\todo[linecolor=magenta,backgroundcolor=magenta!25,bordercolor=magenta,#1]{#2}}

\newcommand{\ud}{\mathrm{d}}

\newcommand{\M}{\mathcal{M}}

\renewcommand{\epsilon}{\varepsilon}

\let\theta\vartheta
\let\phi\varphi
\let\<\langle
\let\>\rangle
\DeclareMathAlphabet{\doba}{U}{msb}{m}{n}

\DeclareMathOperator{\diam}{diam}
\DeclareMathOperator*{\esssup}{esssup}

\def\d{{\text d}}

\newcommand{\onb}[2]{W_{#1}^{#2}}

\newtheorem{lem}{Lemma}
\newtheorem{cor}[lem]{Corollary}
\newtheorem{thm}[lem]{Theorem}
\newtheorem{prop}[lem]{Proposition}

\newtheorem{main}{Main Result}

\theoremstyle{definition}

\newtheorem{rem}[lem]{Remark}

\newcommand{\beq}{\begin{equation}}
\newcommand{\eeq}{\end{equation}}
\newcommand{\N}{\mathbb{N}}
\newcommand{\real}{\mathbb{R}}
\newcommand{\Z}{\mathbb{Z}}
\newcommand{\R}{\mathbb{R}}
\newcommand{\TT}{\mathbb{T}}
\newcommand{\V}{\mathbb{V}}
\newcommand{\rn}{\mathbb{R}^n}

\newcommand{\nat}{\mathbb{N}}

\newcommand{\calT}{\mathcal{T}}
\newcommand{\calM}{\mathcal{M}}
\newcommand{\calP}{\mathcal{P}}

\begin{document}
%%%%%%%%%%%%%%%%%%%%%%%%%%%%%%%%%%%%%%%%%%%%%%%%%%%%%%%%%%%%%%%%%
%\printindex\newpage

\title[Approximation classes for time-stepping AFEM]{Approximation classes for adaptive time-stepping finite element methods}
\author[M. Actis, P. Morin, and C. Schneider]{Marcelo Actis$^\dagger$, Pedro Morin$^\dagger$, and Cornelia Schneider$^{\ast}$}
\thanks{$^{\dagger}$Partially supported by Agencia Nacional de Promoci\'on Científica y Tecnol\'ogica, through grants PICT-2014-2522, PICT-2016-1983, by CONICET through PIP 2015 11220150100661, and by Universidad Nacional del Litoral through grants CAI+D 2016-50420150100022LI. A research stay of Pedro Morin at Universität Erlangen was partially supported by the Simons Foundation and by the Mathematisches Forschungsinstitut Oberwolfach}
\thanks{$^{\ast}${\em (Corresponding author)} Supported by Deutsche Forschungsgemeinschaft (DFG), grant SCHN 1509/1-2.}
\date{\textcolor{black}{\today}}
\subjclass[2010]{Primary 41A25, 65D05; Secondary 65N30, 65N50}
\keywords{}

\maketitle

\begin{abstract}
{We study approximation classes for adaptive time-stepping finite element methods for time-dependent Partial Differential Equations (PDE). We measure the approximation error in $L_2([0,T)\times\Omega)$ and consider the approximation with discontinuous finite elements in time and continuous finite elements in space, of any degree.
As a byproduct we define Besov spaces for vector-valued functions on an interval and derive some embeddings, as well as Jackson- and Whitney-type estimates.
}
\end{abstract}

\tableofcontents

%%%%%%%%%%%%%%%%%%%%%%%%%%%%%%%%%%%%%%%%%%%%%%%%%%%%%%%%%%%%%%%%%
\section{Introduction and main result}
%%%%%%%%%%%%%%%%%%%%%%%%%%%%%%%%%%%%%%%%%%%%%%%%%%%%%%%%%%%%%%%%%

Adaptive time-stepping finite element methods (AFEM) for evolutionary PDE usually lead to a sequence of timesteps and meshes, which yield a partition of the time interval $0=t_0 < t_1 < \dots < t_N = T$ and one triangulation $\calT_i$ for each time interval $[t_{i-1},t_i)$. The complexity of the discrete solution is thus related to the total number of degrees of freedom needed to represent it on the whole interval, which in turn is equivalent to $\sum_{i=1}^N \#\calT_i$.

In this article we study spaces of functions which can be approximated using such time-space partitions with an error of order $\left( \sum_{i=1}^N \#\calT_i\right)^{-s}$ for different $s > 0$. The results that we obtain are similar in spirit to those of~\cite{BDDP02,GM14}, where the spaces corresponding to stationary PDE are considered.

Our goal is not to prove the optimality of AFEM but rather to understand which convergence rates are to be expected for the solutions of evolutionary PDE given their regularity.
In this paper we aim at establishing the first results in this direction, thus at some points we sacrifice generality in order to have a clearer presentation of the basic ideas and set the foundation for further research in this area.

In order to roughly state our main result, we need to introduce some notation, which will be explained in detail later.

Given a polyhedral space domain $\Omega\subset \real^n$, $n \ge 1$, we let $\TT$ denote the set of all triangulations that are obtained through bisection % routine of \cite{Ste08}
from an % properly labeled
initial triangulation $\mathcal{T}_0$ of $\Omega$. For each $\mathcal{T}\in \TT$ we denote by $\#\mathcal{T}$ the number of elements of the partition
%and by $|\calT| = \#\calT - \#\calT_0$, i.e., the number of bisections needed to obtain $\calT$ from $\calT_0$.
%If $d=1$, $\TT$ denotes the set of all partitions of $\Omega = [0,T)$ into sub-intervals that may be obtained by successive bisection of $\calT_0 = \{[0,T)\}$.
%These one-dimensional partitions may be required to satisfy that the length-ratio of neighboring sub-intervals is bounded above by $2$. This requirement is equivalent to asking that neighboring elements belong to consecutive generations, which is the same assumption that holds in the triangulations obtained by the algorithm from~\cite{Ste08}.
%For simplicity, the one-dimensional partition $\{ [0=t_0,t_1), [t_1,t_2),\dots,[t_{N-1},t_N = T) \}$ will be usually stated as $\{ 0=t_0 < t_1 < \dots < t_N = T \}$

For $\calT \in \TT$, we let $\mathbb{V}_{\mathcal{T}}^r$ denote the finite element space of continuous piecewise polynomial functions of fixed order $r$, i.e.,
\[
\mathbb{V}_{\mathcal{T}}^r :=\{g\in C(\overline{\Omega}): \ g\big|_{T}\in \Pi^r\ \text{for all } T\in \mathcal{T}\},
\]
where $\Pi^r$ denotes the set of polynomials of total degree (strictly) less than $r$.

Let $r_1,r_2\in\N$ denote the polynomial orders in time and space, respectively.
Let $\{ 0=t_0<t_1<\ldots<t_N=T \}$ be a partition of the time interval and $\mathcal{T}_1,\ldots, \mathcal{T}_N \in \TT$ be partitions of the space domain $\Omega$, where $\mathcal{T}_i$ corresponds to the subinterval $[t_{i-1},t_i)$, $i=1,\ldots N$.
The time-space partition  as illustrated in  Figure \ref{fig:1}  is then given by
\[
\mathcal{P}=\left(\{0=t_0<t_1<\ldots<t_N=T\}, \{\mathcal{T}_1,\ldots, \mathcal{T}_N\}\right)
\quad \text{with}\quad \# \mathcal{P}=\sum_{i=1}^N \#\mathcal{T}_i
\]
 and $\mathbb{P}$ is the set of all those time-space partitions.
This is the precise kind of time-space partitions produced by time-stepping adaptive methods.

\begin{figure}
\includegraphics[width=9cm]{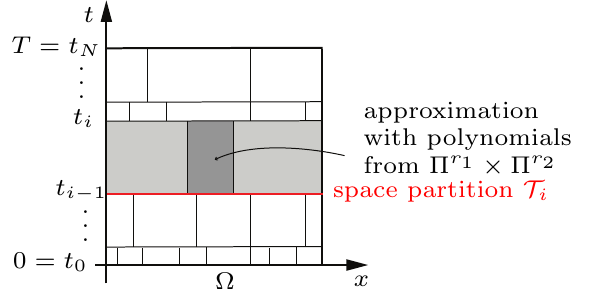}
\caption{ Time-space partition $\mathcal{P}$ }
\label{fig:1}
\end{figure}

The finite element space $\overline{\mathbb{V}}_{\mathcal{P}}^{r_1,r_2}$ subject to such a partition $\mathcal{P}$ is defined as
\begin{align*}
% \overline{\mathbb{V}}_{\mathcal{P}}^{r_1,r_2} &:=\{v: \Omega\times [0,T)\rightarrow \real: \ & v(\cdot, t)\in \mathbb{V}_{\mathcal{T}_i}\ \text{if}\ t\in [t_{i-1},t_i) \ \text{ and } \\
%  &v(x,\cdot)\big|_{[t_{i-1},t_i)}\in \mathcal{Q}^l\text{ for all }x\in \Omega\}\\
\overline{\mathbb{V}}_{\mathcal{P}}^{r_1,r_2} &:=\{g:  [0,T)\times\Omega\rightarrow \real:  g_{[t_{i-1},t_i)\times\Omega} \in \Pi^{r_1} \otimes \mathbb{V}_{\mathcal{T}_i}^{r_2},\text{ for all }i=1,2,\dots,N \},
\end{align*}
i.e., $g \in \overline{\mathbb{V}}_{\mathcal{P}}^{r_1,r_2}$ if and only if $g(t,\cdot)\in \mathbb{V}_{\mathcal{T}_i}^{r_2}$ for all $t\in [t_{i-1},t_i) $ and $g(\cdot,x)\big|_{[t_{i-1},t_i)}\in \Pi^{r_1}$ for all $x\in \Omega$, and all $i=1,2,\dots,N$.
Discrete solutions of adaptive time-stepping methods, e.g. those which use  Discontinuous Galerkin (DG)  in time, belong to spaces of this type.

We define the best $m$-term approximation error by
\[
\overline{\sigma}_m(f)=\inf_{\# \mathcal{P}\leq m}\inf_{g\in \overline{\mathbb{V}}_{\mathcal{P}}^{r_1,r_2}}\|f-g\|_{L_2([0,T)\times\Omega)}.
\]
In this article we measure the error in $L_2([0,T)\times\Omega)$ and leave the general case of $L_p([0,T), L_q(\Omega))$ and other generalizations as future work.

For $s>0$ we define the approximation class $\overline{\mathbb{A}}_s$ as the set those functions whose best $m$-term approximation error is of order $m^{-s}$, i.e.,
\[
\overline{\mathbb{A}}_s:=\{f\in L_2([0,T)\times\Omega): \ \exists c>0 \text{ such that } \overline\sigma_m(f)\leq c\, m^{-s}, \ \forall m\in \nat\}.
\]
Equivalently, we can define $\overline{\mathbb{A}}_s$ through a semi-norm as follows:
\[
\overline{\mathbb{A}}_s:=\{f\in L_2([0,T)\times\Omega): \ |f|_{\overline{\mathbb{A}}_s}<\infty\}\quad \text{with}\quad |f|_{\overline{\mathbb{A}}_s}:=\sup_{m\in \nat}m^s\,\overline\sigma_m(f).
\]
Alternatively, this definition is equivalent to saying that $f\in \overline{\mathbb{A}}_s$ if there is a constant $c$ such that for all $\varepsilon>0$, there exists a time-space partition $\mathcal{P}$ that satisfies
\begin{equation}\label{appr-class-st}
\inf_{g\in \overline{\V}_{\mathcal{P}}^{r_1,r_2}} \|f-g\|_{L_2([0,T)\times\Omega)} \leq c \varepsilon \quad \text{and}\quad \#\calP \leq \varepsilon^{-1/s},
\end{equation}
and $|f|_{\overline{\mathbb{A}}_s}$ is equivalent to the infimum of all constants $c$ that satisfy \eqref{appr-class-st}.

Our main result is stated in terms of Besov spaces, which will be defined in the next section, and reads as follows.

\begin{main}
Let $0 < s_i < r_i$, $i=1,2$,  $0 < q_1 \le \infty$, $1\leq q_2\leq \infty $   with $s_1 > \big(\frac1{q_1}-\frac12\big)_+$ and $s_2 > n (\frac1{q_2}-\frac12\big)_+$.
Then
\[
B^{s_1}_{q_1,q_1}([0,T),L_2(\Omega)) \cap L_2([0,T),B_{q_2,q_2}^{s_2}(\Omega))
\subset \overline{\mathbb{A}}_s
\quad\text{for}\quad  s = \frac{1}{\frac1{s_1}+\frac{n}{s_2}}.
\]
\end{main}

This result is a consequence of Theorem~\ref{space-time-poly}, where, given $f \in B^{s_1}_{q_1,q_1}([0,T),L_2(\Omega) \cap L_2([0,T),B_{q_2,q_2}^{s_2}(\Omega))$, and $\varepsilon > 0$ we construct a time-space partition $\calP$ that satisfies
\[
\# \mathcal{P}\le c_1 \varepsilon^{-\big(\frac1{s_1}+\frac n{s_2}\big)}
\]
and a function $F \in \overline{\V}_{\mathcal{P}}^{r_1,r_2}$ such that
\[
\| f - F \|_{L_2([0,T)\times\Omega)}
\le c_2 \,\varepsilon \,
\left[ | f |_{B^{s_1}_{q_1,q_1}([0,T),L_2(\Omega))} + \| f \|_{L_2([0,T),B_{q_2,q_2}^{s_2}(\Omega))} \right].
\]
Here $B^s_{p,q}(I,X)$ denote Besov spaces of  $X$-valued functions with respective seminorms $|\cdot|_{B^s_{p,q}(I,X)}$, cf.\  Section \ref{subsec-22}.
{It is worth noting that in order to determine the largest spaces, integrability powers $0<p<1$ must be considered. This makes some proofs more complicated than if we were to consider only $p\ge 1$.
}

Our construction is performed in two steps. The first one uses a Greedy algorithm to obtain the partition of the time domain, resorting  in  a Whitney-type estimate for vector-valued functions. That is, we interpret functions in $L_2([0,T)\times\Omega)$ as functions from $[0,T)$ into $L_2(\Omega)$ as is customary in the study of evolutionary PDE, and develop a nonlinear approximation theory for this situation, by revisiting and extending some results from Storozhenko and Oswald~\cite{Sto77,OS78}. This is presented in Section~\ref{sec:generalizedWhitney}, after defining Besov spaces of vector-valued functions in Section~\ref{sec:Besov}.
In Section~\ref{sec:onevariable} we revisit the known results for the stationary case and perform the aforementioned first step by applying the Greedy algorithm to vector-valued functions.
In Section~\ref{sec:timespace} we combine those two results and prove our main result.
We end this article presenting some discussion and comparison of  the approximation classes for space-time discretizations.

We finally mention that we will use $A \lesssim B$ inside some statements, proofs and reasonings in order to denote $A \le c B$ with a constant $c$ that depends on the parameters indicated in the corresponding statement. As usual, $A\simeq B$ means $A\lesssim B$ and $B \lesssim A$.

\section{Besov spaces of vector-valued functions}%
\label{sec:Besov}

The goal of this section is to define and understand some properties of Besov spaces of functions from a real interval $I$ into a Banach space.
From now on, we let $X$ be a separable Banach space with norm $\| \cdot \|_X$.

We first introduce the moduli of smoothness and state and prove some of their properties, which are analogous to those corresponding to the case of real-valued functions. Afterwards we define the corresponding Besov spaces and state and prove some embeddings.

\subsection{Moduli of smoothness of vector-valued functions on an interval} % and their basic properties}

We start this section by providing new definitions of moduli of smoothness for vector-valued functions, which are analogous to the ones already known for real-valued functions, and stating and proving some of their basic properties.

It is worth mentioning that there is a forerunner regarding moduli of smoothness and Whitney-type estimates of vector-valued functions, cf. \cite{DF90}. However,  our definition (which is an immediate generalization of the classical moduli of real-valued functions) differs from the one given in \cite{DF90} (which is more elaborate and tricky). In particular, in \cite{DF90} a duality approach is used between the given Banach space and its dual in order to reduce the definitions and results for abstract functions to real ones. But there is a price to pay:  the results are restricted to the set of bounded functions. Therefore even classical Banach spaces like $L_p(I,X)$ cannot be considered entirely.

%From now on, $X$ denotes a separable Banach space.
Given $0<p\le\infty$, a real interval $I${$=[a,b)$ with $|I|=b-a$}, and a function $f : I \to X$, we say that $f \in L_p(I,X)$ if $f$ is measurable and $\| f \|_{L_p(I,X)} := \Big(\int_I \| f(t) \|_X^p \d t\Big)^{1/p} < \infty$ if $p<\infty$ and $\| f \|_{L_\infty(I,X)} = \esssup_{t \in I} \| f(t) \|_X $.
For such a function $f$, $r\in\N$ and $0<|h|<\frac{{|I|}}r$, the $r$-th order difference $\Delta_h^rf : I_{rh} \to X$ is defined as
\[
\Delta_h^r f(t) = \sum_{i=0}^r {r \choose i} (-1)^{r-i} f(t+ih),\qquad t \in I_{rh} := \{ t \in I : t+rh \in I\},
\]
which clearly satisfies $\Delta_h^r f = \Delta_h \Delta_h^{r-1} f$ and $\| \Delta_h^r f \|_{L_p(I_{rh})}^{\min\{1,p\}} \le 2 \| \Delta_h^{r-1} f \|_{L_p(I_{(r-1)h})}^{\min\{1,p\}}$, understanding that $\Delta_h f = \Delta_h^1 f$ and $\Delta_h^0 f = f$.

The modulus of smoothness is defined as
\beq\label{omega1}
\omega_r(f,I,u)_p := \sup_{0<|h| \leq u}\|\Delta_h^r f \|_{L_p(I_{rh},X)}
= \sup_{0<h \leq u}\|\Delta_h^r f \|_{L_p(I_{rh},X)}
,\qquad u > 0,
\eeq
which is clearly increasing as a function of $u$,
and the \emph{averaged} modulus of smoothness is defined, for $u>0$, as
\beq\label{w1}
w_r(f,I,u)_p:=\left(\frac{1}{2u}\int_{-u}^u\|\Delta_h^r f\|^p_{L_p(I_{rh},X)} \, \d h\right)^\frac 1p
=\left(\frac{1}{u}\int_{0}^u\|\Delta_h^r f\|^p_{L_p(I_{rh},X)} \, \d h\right)^\frac 1p.
\eeq

The well-known definitions for $f : \Omega \to \R$, with $\Omega$ a domain of $\R^n$, $n \ge 1$, are as follows. For $h \in \R^n$,
the domain of $\Delta_h^r f$ is the set $\Omega_{rh} := \{ x \in \Omega : x, x+h, \dots,x+rh \in \Omega\}$, and the moduli of smoothness $\omega_r(f,\Omega,u)_p$, $w_r(t,\Omega,u)_p$ are defined  for $u>0$ via
\begin{align}
\omega_r(f,\Omega,u)_p &:= \sup_{0< |h| \leq u}\|\Delta_h^r f \|_{L_p(\Omega_{rh})} ,
 \label{mod-smooth-B}
\\
w_r(f,\Omega,u)_p&:=\left(\frac{1}{(2u)^n}\int_{[-u,u]^n} \|\Delta_h^r f\|^p_{L_p(\Omega_{rh})} \, \d h\right)^\frac 1p. \notag
\end{align}

As a consequence of the fact that
$ \Delta_{mh}^1 f(x) = \sum_{i=0}^{m-1} \Delta_h^1 f(x+ih)$, for $m\in\N$,
we can prove by induction  $\| \Delta_{mh}^r f \|_{L_p(A_{rmh})} \le m^r \| \Delta_h^r f \|_{L_p(A_{rh})}$, for $A = I$ or $A = \Omega$ (for details see~\cite[Sect.~3.1]{PP87}).
As an immediate consequence of this,
\beq\label{homogeneity}
\omega_r(f,A,mu)_p^{\min\{1,p\}} \le m^r \omega_r(f,A,u)_p^{\min\{1,p\}}, \quad u > 0.
\eeq
From the properties stated above, we  have
\beq\label{inductivebound}
w_{r+1}(f,A,u)^{\min\{1,p\}}_p \le 2 w_r(f,A,u)^{\min\{1,p\}}_p.
\eeq
Finally, we notice that if $f : [a,b) \to X$ and $\hat f : [0,1) \to X$ with $\hat f(t) = f(a+t(b-a))$, then, for $u>0$
\beq\label{scaling}
\begin{split}
\omega_r(f,[a,b),u)_p &= (b-a)^{1/p} \, \omega_r(\hat f, [0,1),  (b-a)^{-1} u)_p,
\\ %\quad\text{and}\quad
w_r(f,[a,b),u)_p &= (b-a)^{1/p} \, w_r(\hat f, [0,1),  (b-a)^{-1}u)_p.
\end{split}
\eeq

Now we prove that the two moduli of smoothness $w_r$ and $\omega_r$ as defined above in~\eqref{omega1} and~\eqref{w1} are equivalent.
This result is well-known and proved for real-valued functions in~\cite[Lem. 6.5.1]{DL93}. The proof for vector-valued functions is analogous and we sketch it here for completeness.

\begin{lem}\label{lem-equiv}
Given $0<p<\infty$ and $r \in \N$ the two definitions of moduli of smoothness $w_r(\cdot,\cdot,\cdot)_p$ and $\omega_r(\cdot,\cdot,\cdot)_p$ are equivalent, more precisely
\begin{align*}
    w_r(f,I,u)_p %=\left(\frac{1}{t}\int_{0}^t\|\Delta_h^r f\|^p_{L_p(I,X)} \, \d h\right)^\frac 1p
 \le
%\sup_{|h|\leq t}\|\Delta_h^r f \|_{L_p(I,X)} =
\omega_r(f,I,u)_p
\le c w_r(f,I,u)_p,
\end{align*}
for all $f \in L_p(I,X)$, $I = [a,b)$ and $0 < u < |I|/r$, where the constant $c$ depends only on $r$ and $p$, but is otherwise independent of $f$, $I$, and $u$.
%\textcolor{red}{We leave out $X$ in the notation of the moduli, since it is irrelevant in this context. }
\end{lem}

\begin{proof}
The fact that $w_r(f,I,u)_p\leq \omega_r(f,I,u)_p$ is obvious. Therefore, it remains to prove the converse inequality. We prove the result for the reference situation of $I = [0,1)$, the general case follows by scaling using~\eqref{scaling}.

We use the reproducing formula
\beq \label{repr-form-1}
\Delta_h^rf(t)=\sum_{l=1}^r(-1)^l{r\choose l}\left[\Delta_{ls}^r f(t+lh)-\Delta_{h+ls}^r f(t)\right],
\eeq
which holds if $t\in [0,1-rh]$ and
\[
t+lh+rls\leq  1 \quad \text{and}\quad t+rh+rls\leq 1.
\]
This together yields the range $t\in [0,1-rh-r^2s]$.
Formula \eqref{repr-form-1} is proved by induction, starting with the observation that %: Let $r=1$, then
\begin{align*}
\Delta^1_hf(t)
&=f(t+h)-f(t)\\
&=f(t+h)-f(t+h+ls)+f(t+h+ls)-f(t)\\
&= - \big[\Delta^1_{ls}f(t+h)-\Delta^1_{h+ls}f(t) \big].
\end{align*}
%\textcolor{red}{do remaining induction}

We now consider $0<h \le u\leq \frac{1}{4r}$ and $0\leq t\leq \frac 12$. This gives us the upper bound $s<\frac{1}{4r^2}$. Integrating formula \eqref{repr-form-1} yields
\[
\int_0^{1/2}\|\Delta^r_h f(t)\|_X^p\ud t\lesssim \sum_{l=1}^r \int_0^{1/2}\|\Delta_{ls}^r f(t+lh)\|_X^p\ud t +\int_0^{1/2}\|\Delta^r_{h+ls}f(t)\|_X^p\ud t.
\]
Thus, setting $I_{-}:=[0,1/2]$ and averaging over $s\in \left[0,u\right]$ gives
\begin{align}
    \|\Delta^r_h f&\|_{L_p(I_{-},X)}^p
    \\
    &\lesssim \sum_{l=1}^r\frac{1}{u}\left[
    \int_0^u \int_{I_{-}}\|\Delta_{ls}^r f(t+lh)\|_X^p\ud t\ud s
    + \int_0^u \int_{I_{-}}\|\Delta_{h+ls}^r f(t)\|_X^p\ud t\ud s
    \right]\notag\\
     &= \sum_{l=1}^r\frac{1}{lu}\Bigg[
    \int_0^{lu} \int_{I_{-}}\|\Delta_{h'}^r f(t+lh)\|_X^p\ud t\ud h'
    + \int_h^{h+lu} \int_{I_{-}}\|\Delta_{h'}^r f(t)\|_X^p\ud t\ud h'\Bigg]\notag\\
     &\lesssim \sum_{l=1}^r\frac{1}{(r+1)u}
    \int_0^{(r+1)u} \|\Delta_{h'}^r f\|_{L_p(I,X)}^p \ud h'\notag \\
    &\leq w_r(f,I,(r+1)u)_p^p, \label{est-07}
\end{align}
where in the second step we used the substitution $h':=ls$ in the first and $h':=h+ls$ in the second integral.
By symmetry, we also have that $\|\Delta^r_{-h} f\|_{L_p(I_{+},X)}^p \le w_r(f,I,(r+1)u)_p^p$ with $I_+ = [1/2,1]$.
Taking the supremum w.r.t. $0<h\leq u$ on both sides we arrive at
\begin{align*}
    \omega_r(f,I,u)_p
    &\lesssim w_r(f,I, (r+1)u)_p
\end{align*}
Using~\eqref{homogeneity} we obtain
\begin{align*}
    \omega_r(f,I,(r+1)u)_p
     \lesssim \omega_r(f,I,u)_p
     \lesssim w_r(f,I,(r+1)u)_p,
\end{align*}
which completes the proof.
\end{proof}

\subsection{Besov spaces and embeddings}
\label{subsec-22}

Using the generalized modulus of smoothness defined in the previous subsection, we introduce the Besov spaces $B^s_{p,q}(I,X)$, $s > 0$, $0<p,q\le \infty$,  which contain   all functions $f\in L_p(I,X)$ such that for $r:=\lfloor s\rfloor+1$ the quasi-seminorm
\beq \label{seminorm-B}
\begin{split}
|f|_{B^{s}_{p,q}(I,X)}&:=
%\begin{cases}
\displaystyle\left(\int_0^{|I|/r}\left[u^{-s}\omega_r(f,I,u)_p\right]^q\frac{\ud u}{u}\right)^{1/q}<\infty
,\qquad  0<q<\infty, \\ %\\
|f|_{B^{s}_{p,\infty}(I,X)} &:=
\sup_{0<u<|I|/r}u^{-s}\omega_r(f,I,u)_p < \infty. %, & &q=\infty,
%\end{cases}
\end{split}
\eeq
%(with the usual modification if $q=\infty$).
Moreover, a quasi-norm for $B^s_{p,q}(I,X)$ is given by
\beq\label{norm-B}
\|f\|_{B^s_{p,q}(I,X)}:=\|f\|_{L_p(I,X)} + |f|_{B^s_{p,q}(I,X)} ,
\eeq
which is a norm whenever $1\leq p,q\leq \infty$.

%\textcolor{red}{\item One should obtain a quasi-norm equivalent to \eqref{norm-B} for $B^s_{p,q}(I,X)$ by replacing $\omega_r$ by $\omega_k$ for any $k>s$ (however for this one needs to check Marchaud's inequality for our setting). }

\begin{rem}\label{rem-equiv}
One can replace the integral $\int_0^{|I|/r}$ by $\int_0^1$ if $|I|<\infty$ and still get an equivalent norm.
{More precisely,
\[
\int_0^{|I|/r} [u^{-s}\omega_r(f,I,u)_p]^q\frac{\ud u}{u}
\simeq \int_0^1 [u^{-s}\omega_r(f,I,u)_p]^q\frac{\ud u}{u}
\]
with equivalence constants that depend only on $s$, $r$, $p$, $q$, but are otherwise independent of $f$ and $|I|$ as $|I| \to 0$.
}

{
We prove this claim for $0<q<\infty$, the case $q=\infty$ is analogous.
If $|I|/r<1$ then, on the one hand, $\int_0^{|I|/r} [u^{-s}\omega_r(f,I,u)_p]^q\frac{\ud u}{u}
\le \int_0^1 [u^{-s}\omega_r(f,I,u)_p]^q\frac{\ud u}{u}$.
On the other hand, $\omega_r(f,I,u)_p=\omega_r(f,I,|I|/r)_p$, when $u\geq |I|/r$.
Therefore, using~\eqref{homogeneity} and the monotonicity of $\omega_r(f,I,\cdot)_p$,
\begin{align*}
\int_{|I|/r}^1 [u^{-s}\omega_r(f,I,u)_p]^q\frac{\ud u}{u}
&= \omega_r(f,I,|I|/r)_p^q \int_{|I|/r}^1 u^{-sq-1} \ud u
\\
&\lesssim \omega_r(f,I,|I|/(2r))_p^q  \, (|I|/r)^{-sq}
\\
&\lesssim  \omega_r(f,I,|I|/(2r))_p^q \int_{|I|/(2r)}^{|I|/r} u^{-sq-1} \ud u
\\
&\le \int_{|I|/(2r)}^{|I|/r}[u^{-s}\omega_r(f,I,u)_p]^q\frac{\ud u}{u},
\end{align*}
which yields the second inequality for the case $|I|/r<1$.
}

{
If $|I|/r > 1$, trivially $\int_0^1 [u^{-s}\omega_r(f,I,u)_p]^q\frac{\ud u}{u} \le \int_0^{|I|/r} [u^{-s}\omega_r(f,I,u)_p]^q\frac{\ud u}{u}$. Besides, using again~\eqref{homogeneity} and the monotonicity of $\omega_r(f,I,\cdot)_p$,
}
\[
\omega_r\left(f,I,\frac 12\right)_p\leq \omega_r(f,I,u)_p\leq \omega_r\left(f,I,\frac{|I|}r\right)_p\lesssim {|I|^r} \omega_r\left(f,I,\frac 12\right)_p, \quad \frac 12\leq u\leq |I|/r.
\]
Hence,
\[
\int_1^{|I|/r}[u^{-s}\omega_r(f,I,u)_p]^q\frac{\ud u}{u}
\lesssim {|I|^{rq}}\omega_r\left(f,I,\frac 12\right)_p^q
\lesssim {|I|^{rq}}\int_{\frac 12}^1 [u^{-s}\omega_r(f,I,u)_p]^q\frac{\ud u}{u},
\]
which proves the claim for the case $|I|/r > 1$.
\end{rem}

\begin{rem}
Our definition for the Besov spaces above is in good agreement with the standard case: When $f:\Omega\rightarrow \real$, with $\Omega$ a domain of $\real^n$, the usual Besov spaces $B^s_{p,q}(\Omega)$ are defined as those subspaces containing all functions  $f\in L_p(\Omega)$ for which
\beq %\label{seminorm-B-2}
|f|_{B^{s}_{p,q}(\Omega)}:=
%\begin{cases}
\displaystyle\left(\int_0^{\diam(\Omega)}\left[u^{-s}\omega_r(f,\Omega,u)_p\right]^q\frac{\ud u}{u}\right)^{1/q}<\infty
%, & 0<q<\infty,%\\
%\displaystyle \quad \sup_{0<t<|I|}t^{-s}\omega_r(f,I,t)_p, & q=\infty,
%\end{cases}
\eeq
(with the usual modification if $q=\infty$) and $r = \lfloor s \rfloor + 1$. Here the modulus of smoothness involved is the usual one given in    \eqref{mod-smooth-B}.
The space $B^s_{p,q}(\Omega)$ is then quasi-normed via
$\|f\|_{B^s_{p,q}(\Omega)}:=\|f\|_{L_p(\Omega)}+|f|_{B^s_{p,q}(\Omega)}.
$ For more information on these spaces we refer to \cite{DL93, T83}.
\end{rem}

Later on it will be useful for us to discretize the quasi-seminorm \eqref{seminorm-B} as follows.

\begin{lem}\label{discrete-seminorm}
The  quasi-seminorm \eqref{seminorm-B} for $B^s_{p,q}(I,X)$ is equivalent to
\beq \label{seminorm-B-2}
\begin{split}
|f|_{B^{s}_{p,q}(I,X)}^* &:=
%\begin{cases}
\left(\sum_{k=0}^{\infty}\left[2^{ks}\omega_r(f,I,2^{-k})_p\right]^q\right)^{1/q}, \qquad  0<q<\infty, \\
|f|_{B^{s}_{p,\infty}(I,X)}^* &:=
\sup_{k\geq 0}2^{ks}\omega_r(f,I,2^{-k})_p, %\  q=\infty,
%\end{cases}
\end{split}
\eeq
with constants of equivalence independent of $f$ and $I$ as $|I| \to 0$.
\end{lem}

\begin{proof} The proof follows along the lines of the standard case, which may be found in \cite[p.~56]{DL93}. Using \eqref{homogeneity} with $m=2$ and the monotonicity of $\omega_r(f,I,\cdot)$ we see that for $u\in [2^{-k-1}, 2^{-k}]$ it holds
\begin{align*}
  2^{-r}\left(2^{ks}\omega_r(f,I,2^{-k})_p\right)^{\min\{1,p\}}  &\leq \left(u^{-s}\omega_r(f,I,u)\right)^{\min\{1,p\}}_p\\
  &\leq \left(2^{(k+1)s}\omega_r(f,I,2^{-k})_p\right)^{\min\{1,p\}}.
\end{align*}
Raising all terms of the inequality to the power  ${\frac{1}{\min\{1,p\}}}$ we obtain
\[
u^{-s}\omega_r(f,I,u)_p \simeq 2^{ks}\omega_r(f,I,2^{-k})_p \qquad \text{for} \qquad u\in [2^{-k-1}, 2^{-k}].
\]
Hence, since $\int_{2^{-k-1}}^{2^{-k}}\frac{\ud u}{u}=\ln 2\simeq 1$  we get 
\[
\left(\int_{2^{-k-1}}^{2^{-k}}[u^{-s}\omega_r(f,I,u)_p]^q\frac{\ud u}{u}\right)^{1/q}\simeq 2^{ks}\omega_r(f,I,2^{-k})_p.
\]
This completes the proof for $0<q<\infty$ taking into account Remark \ref{rem-equiv}, after adding all terms for $k\ge0$. The case $q=\infty$ is analogous.
\end{proof}

%\begin{comment}

\subsubsection{Embedding results}
Before we provide some embeddings for the scale $B^s_{p,q}(I,X)$ needed later on,
let us briefly recall what is known concerning the Besov  spaces  $B^s_{p,q}(\Omega)$.
%Since we deal with both spaces simultaneously, we shall simply use the notation $B^s_{p,q}$ to denote the scalar as well as  the $X$-valued case in what follows.

\begin{prop} \label{prop-emb-B}
Let   $s>0$ and $0<p,q\leq\infty$.
\begin{itemize}
\item[(i)]
Let $0<\varepsilon<s$, $0< \nu\leq\infty$, and $q\leq \theta\leq\infty$, then
\[
B^{s}_{p,q}(\Omega) \hookrightarrow B^{s-\varepsilon}_{p,\nu}(\Omega)\qquad\text{and}\qquad
B^s_{p,q}(\Omega) \hookrightarrow B^s_{p,\theta}(\Omega).
\]
\item[(ii)] (Sobolev-type embedding)  \
Let $0<\sigma< s$ and $p< \tau $ be such that
\beq
s-\frac{n}{p} \geq  \sigma-\frac{n}{\tau},
\label{delta-B}
\eeq
then
\beq\label{sob-emb-B}
B^s_{p,q}(\Omega) \hookrightarrow B^\sigma_{\tau,\theta}(\Omega),
\eeq
where $0<\theta\leq \infty$ and, additionally,  $q\leq \theta$ if an equality holds in \eqref{delta-B}. Moreover, in the limiting case when $\sigma=0$
%\item[(iii)] Let $0<s<\frac{n}{p}$,
and $\theta$ is such that
\beq
s-\frac{n}{p}\geq -\frac{n}{\theta},
\label{delta-2-B}
\eeq
we have
\beq\label{Lim-emb-B}
B^s_{p,q}(\Omega) \hookrightarrow L_{\theta}(\Omega),
\eeq
where again $q\leq \theta$ if an equality holds in \eqref{delta-2-B}.
\item[(iii)] If the domain $\Omega\subset \rn$ is bounded, then for $\tau\leq p$  we have the  embedding
\beq
B^s_{p,q}(\Omega)\hookrightarrow B^s_{\tau,q}(\Omega). %, \qquad \tau\leq p.
\eeq
%where $\tau\leq p$.
\end{itemize}
\end{prop}

%\textcolor{red}{Think about the proof! Do wee need Marchaud inequality for (ii)? Add limiting case, i.e., embedding into $L_q$!}
%\todo[inline]{Pedro cannot find these embeddings explicitely in the references below. Is there a place where we could find them almost as we write them here?}

%\begin{figure}[H]
\noindent
\begin{minipage}{0.44\textwidth}
\begin{rem}
\begin{itemize}
  \item[(i)]
The above results can be found in \cite[\S~2.10, 12.8]{DL93}, \cite[Thm.~1.15]{HS09}, and \cite{BS88}.  \\
In the interpolation diagram aside we have  illustrated  the area of possible embeddings of a fixed original space $B^s_{p,q}(\Omega)$ into spaces $B^{\sigma_1}_{\tau_1,\nu_1}(\Omega)$ and $B^{\sigma_2}_{\tau_2,\nu_2}(\Omega)$. The lighter shaded area corresponds to the additional embeddings we have if the underlying domain $\Omega$ is  bounded.
\end{itemize}
\end{rem}
\end{minipage}\hfill \begin{minipage}{0.55\textwidth}
\includegraphics[width=\textwidth]{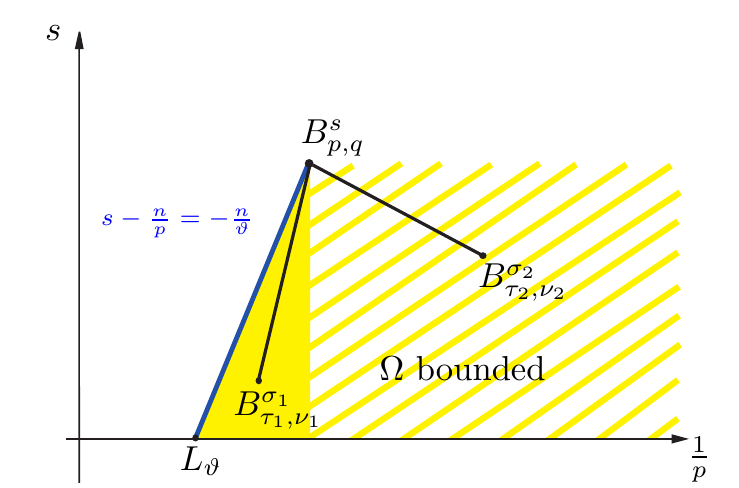}
\captionof{figure}{Embeddings for  $B^s_{p,q}(\Omega)$}
\end{minipage}\\

\begin{itemize}
    \item[(ii)] In the non-limiting case (corresponding to the strict inequality in \eqref{delta-B} and \eqref{delta-2-B}) the embeddings in Proposition \ref{prop-emb-B} are known to be compact. In particular, for  $\alpha>0$ and  $p<\tau$,  the embeddings $B^{s+\alpha}_{p,p}(\Omega)\hookrightarrow B^{s}_{\tau, \tau}(\Omega)$ ($s>0$) and $B^{\alpha}_{p,p}(\Omega)\hookrightarrow L_{\tau}(\Omega)$ ($s=0$) are compact if, and only if,
\[\alpha-\frac{n}{p}>-\frac {n}{\tau}.
\]
\end{itemize}

For the scale $B^s_{p,q}(I,X)$  there are counterparts of the embeddings from Proposition~\ref{prop-emb-B}.

\begin{prop}\label{prop-emb-BX}
Assume  $s>0$ and $0<p,q\leq\infty$.
\begin{itemize}
\item[(i)]
Let $0<\varepsilon<s$, $0< \nu\leq\infty$, and $q\leq \theta\leq\infty$, then
\[
B^{s}_{p,q}(I,X) \hookrightarrow B^{s-\varepsilon}_{p,\nu}(I,X)\qquad\text{and}\qquad
B^s_{p,q}(I,X) \hookrightarrow B^s_{p,\theta}(I,X).
\]
%\item[(ii)] (Limiting embedding)  \
%Let $X = L_2(\Omega)$ with $\Omega$ a bounded Lipschitz domain of $\R^n$. Let $p<2$ be such that
%\beq
%s-\frac1p = -\frac12,
%\label{delta-BX}
%\eeq
%we then have
%\beq
%B^s_{p,p}(I,L_2(\Omega)) \hookrightarrow L_2(I, L_2(\Omega))=L_2(I\times \Omega).
%\eeq
%%where $0<r\leq \infty$ and, additionally, %$q\leq r$ if $A=B$ and '$=$' holds in %\eqref{delta-B}.
\item[(ii)] If the time interval $I$ is bounded, then for $\tau\leq p$   we have the  embedding
\beq
B^s_{p,q}(I,X)\hookrightarrow B^s_{\tau,q}(I,X). %, \qquad \text{for}\quad \tau\leq p.
\eeq
%where $\tau\leq p$.
%Let $\sigma_p <s<\frac{n}{p}$, $q\leq u \leq\infty$, and $r$ such that
%\beq
%s-\frac{n}{p}=-\frac{n}{r},
%\label{delta-2-B}
%\eeq
%then
%\[
%\B(\rn) \hookrightarrow L_{r,u}(\rn).
%\]
\end{itemize}
\end{prop}

\begin{proof}
The embeddings in (i) and  (ii) can be proven as in the standard case, using the discrete version of the seminorm for Besov spaces, i.e.,
\[
|f|_{B^s_{p,q}(I,X)} \simeq |f|_{B^s_{p,q}(I,X)}^* = \left( \sum_{k=0}^{\infty} [2^{ks}\omega_r(f,2^{-k})_p]^q\right)^{\frac 1q},
\quad 0 < q < \infty,
\]
with the analogous one for $q = \infty$.
Indeed, the second embedding in (i) is just a consequence of the monotonicity of the $\ell_q$ sequence spaces, i.e., $\ell_q\hookrightarrow \ell_{\theta}$ for $q\leq \theta$.
The first embedding for $q\leq \nu$ is also clear since $2^{k(s-\varepsilon)}\leq 2^{ks}$. If $\nu<q$ one uses H\"older's inequality with $\frac q\nu>1$, which gives the desired result. \\
Moreover, (ii)  follows immediately  since for $\tau\leq p$ and  $|I|<\infty$ we have  $
L_p(I,X)\hookrightarrow L_{\tau}(I,X).
$
\end{proof}

\begin{comment}
\begin{cor}\label{cor-emb-BX}
If $I$ is a bounded interval, we have that
\[
B_{p,p}^s(I,L_2(\Omega)) \hookrightarrow L_2(I\times\Omega),
\]
whenever $s > 0$ and $s-1/p \pedro{\ge} -1/2$.
\end{cor}

\begin{proof}
The result for $p < 2$ is contained in Proposition~\ref{prop-emb-BX}~(ii), using the first embedding of (i). If $p \ge 2$ and $s-1/p > -1/2$ we let $\epsilon > 0$ be such that $s-\epsilon-1/p > -1/2$ and use Proposition~\ref{prop-emb-BX}~(i) and~(iii) to see that
\[
B_{p,p}^s(I,X) \hookrightarrow B_{p,2}^{s-\epsilon}(I,X)
\hookrightarrow B_{2,2}^{s-\epsilon}(I,X).
\]
The last space is clearly contained in $L_2(I,X) = L_2(I\times\Omega)$.
\end{proof}
\end{comment}

\begin{rem}
The counterpart of the limiting embedding \eqref{Lim-emb-B} in Prop. \ref{prop-emb-B}(ii) is derived in Corollary \ref{cor-gen-whitney} as an application of our generalized Whitney's estimate presented in Proposition~\ref{prop-gen-whitney}. Moreover, the  Sobolev-type embeddings as stated  in Prop. \ref{prop-emb-B}(ii), formula \eqref{sob-emb-B}, should also hold. The proof in the standard case, cf.  \cite[\S~ 12.8]{DL93}, involves spline representations for Besov spaces, which we have not provided for our generalized setting so far. This is out of the scope of the present paper.
\end{rem}

\section{Jackson- and Whitney-type theorems for vector-valued functions}
\label{sec:generalizedWhitney}

{In this section we prove Jackson- and Whitney-type theorems for functions defined on an interval, but valued on a Banach space. Some proofs are rather technical, and analogous to the ones presented for scalar-valued functions in \cite{Sto77, OS78}.
}

{Let us mention that regarding Jackson's theorem there is a proof for $1\le p \le \infty$, which is based on the $K$-functional method of interpolation~\cite[\S3.5]{PP87} and seems extendable to vector-valued functions.
There is an alternative proof in~\cite[Thm.~7.1]{PP87}, which holds for $0<p\le \infty$ and avoids all the technicalities from~\cite{Sto77, OS78}. However, it is based on a contradiction argument and does not work in the vector-valued case, or at least we could not generalize it to the infinite-dimensional setting.}

The proof of Whitney's theorem that we present below in Section~\ref{S:Whitney} follows the steps from~\cite[Sect.~6.1]{DeV98}. In order to do it, we need an equivalence of $L_p$-norms for vector-valued polynomials, which is contained in Lemma~\ref{lem-scaling-aux} and Corollary~\ref{cor-scaling-lpq}. After proving Whitney's estimate in $B^s_{q,q}(I,X)\cap L_p(I,X)$ in Proposition~\ref{prop-gen-whitney} we obtain the embedding $B^s_{q,q}(I,X) \subset L_p(I,X)$, and arrive at Whitney's estimate in $B^s_{q,q}(I,X)$.

\subsection{Jackson's estimate}

The goal of this section is to prove a Jackson-type estimate, which is stated below in~Theorem~\ref{thm-jackson} and requires some definitions.
%. Before stating it we need some definitions.

Given a separable Banach space $X$, $r \in \N$, and an interval $I = [a,b)$, we denote by $\mathbb{V}^r_{I,X}$ the space of $X$-valued polynomials of order $r$ w.r.t.\ time, which we define as follows:
\beq \label{Vr}
	\mathbb{V}^r_{I,X}:=\bigg\{ P:I \to X , P(t)=\sum_{j=1}^{r}\ell_j^r(t) P_j: \ P_j\in X, \ t\in I\bigg\},
	\eeq
	with $\ell_j^r$ the usual (scalar-valued) Lagrange basis functions
	\beq \label{Lagrange}
\ell_j^r(t)=\prod_{i\neq j}\frac{t-t_i}{t_j-t_i}\qquad \text{for } t_j = a + (j-1) \frac{b-a}{r-1},
\qquad j=1,2,\dots,r.
\eeq
Notice that any basis for the space $\Pi^r$ of scalar-valued polynomials in $\real$, such as $1,t,t^2,\dots,t^{r-1}$, leads to the same space $\mathbb{V}^r_{I,X}$.

The main result of this section is the following.

\begin{thm}[Jackson's Theorem]\label{thm-jackson}
Let $0<p\le\infty$ and $r \in \N$. Then there exists a constant $c>0$ such that for any interval $I$ and every $f \in L_p(I,X)$, there exists a vector-valued polynomial $P_r \in \V_{I,X}^r$, which satisfies
\beq\label{Jackson0}
\|f-P_r\|^p_{L_p(I,X)}\leq c \, w_r(f,I,h)_p^p\qquad \text{with}\qquad h=\frac{|I|}{2r}.
\eeq
In other words, there exist $a_0$, $a_1$, \dots, $a_{r-1}\in X$ such that, if $P_r(t)=a_0+a_1 t+\ldots + a_{r-1}t^{r-1}$, then  \eqref{Jackson0}  holds.
\end{thm}

Due to the homogeneity~\eqref{homogeneity} and the equivalence of Lemma~\ref{lem-equiv}, Jackson's estimate can also be stated as:
\beq\label{Jackson}
E_r(f,I)_p
:= \inf_{P_r \in \V_{I,X}^r} \|f-P_r\|^p_{L_p(I,X)}
\leq c \, w_r(f,I,|I|)_p^p, \qquad \forall f \in L_p(I,X).
\eeq

In order to prove this estimate, we need several auxiliary lemmas, {which are rather technical, and analogous to the ones proved} for scalar-valued functions in \cite{Sto77, OS78}. We generalize them to our setting.
The basic idea is to first study periodic functions and their higher order differences, and then relate them to differences of the functions we are actually interested in.
%\cosc {It is worth mentioning that the contradiction argument used in~\cite{PP87} to prove this result in the scalar-valued case, which would avoid all the technicalities of this section, does not work in the vector-valued case, or at least we could not generalize it to the infinite-dimensional case.
%} \scco 

Let $f:[a,b)\rightarrow X$ be an $X$-valued function and
$f^{\ast}$ denote its periodic continuation with period $d:=b-a$, i.e.,
\[
f^{\ast}(t)= f(t-\ell d), \quad \text{where $\ell\in \Z$ is such that $t-\ell d\in [a,b)$}.
\]
Moreover,  for  $0<p<\infty$ and $k\in \nat$  consider the integrals
\begin{align}
    I^{\ast}_{p,k}(h)&:=\int_a^b \|\Delta^k_hf^{\ast}(t)\|_X^p\ud t =\int_0^d \|\Delta^k_hf^{\ast}(t)\|_X^p\ud t,\label{I_p_1}\\
     I_{p,k}(h)&:=\int_a^{b-kh} \|\Delta^k_hf(t)\|_X^p\ud t.\label{I_p_2}
\end{align}
Note that we do not emphasize on the fact that the expressions $I^{\ast}_{p,k}(h)$ and $I_{p,k}(h)$ also depend on the functions $f$ and $f^{\ast}$, respectively, since it will always be clear from the context which function we deal with.

We start with the following result showing how the best approximation of some function $f\in L_p(I,X)$ by a constant $a_0\in X$ can be bounded using first differences of its periodic continuation $f^{\ast}$.

\begin{lem}\label{lem-aux-1}
Let $0<p<\infty$ and $f\in L_p(I,X)$.
There exists $a_0\in X$ such that
\[
%\inf_{a_0\in X}
\|f-a_0\|^p_{L_p(I,X)}
\leq \frac 1d \int_0^d I^{\ast}_{p,1}(y)\ud y.
\]
\end{lem}

\begin{proof} We show how  to construct $a_0\in X$ satisfying the desired inequality. Let $f^{\ast}$ denote the $d$-periodic continuation of $f$.  We make the following easy observation,
\begin{align*}
    \inf_{a_0\in X}\|f-a_0\|_{L_p(I,X)}^p
%    &= \inf_{a_0\in X}\int_a^b \|f(t)-a_0\|_{X}^p\ud t\\
    &=\inf_{a_0\in X}\int_0^d \|f^{\ast}(t)-a_0\|_{X}^p\ud t\\
    &= \inf_{a_0\in X}\int_0^d \|f^{\ast}(t+y)-a_0\|_{X}^p\ud t\\
    &\leq \int_0^d \|f^{\ast}(t+y)-f^{\ast}(y)\|_{X}^p\ud t, \quad \text{for any $y \in [0,d)$}.
\end{align*}
Now using the fact that $f^{\ast}$ is $d$-periodic and the left-hand side does not depend on $y$, integration from $0$ to $d$ w.r.t. $y$ yields
\begin{align*}
    \inf_{a_0\in X}\|f-a_0\|_{L_p(I,X)}^p
    & \leq \frac 1d\int_0^{d}\int_0^{d}\|f^{\ast}(t+y)-f^{\ast}(y)\|_X^p\ud t\ud y\\
   & = \frac 1d\int_a^b\int_a^b\|f(t)-f(y)\|_X^p\ud t\ud y
   = \frac 1d \int_a^b g(y)\ud y,
\end{align*}
where in the last line we put $g(y):=\int_a^b\|f(t)-f(y)\|_X^p\ud t$. Note that the set   $S$ defined as
\[
S:=\Big\{ z\in [a,b): \ g(z)\leq \frac 1d \int_a^b g(y)\ud y\Big\},
\]
is non-empty.  Therefore,
 taking $z\in S$ and putting $a_0:=f(z)$ we obtain
\begin{align*}
   \|f-a_0\|^p_{L_p(I,X)}
   &= \int_a^b \|f(t)-f(z)\|_{X}^p \ud t =g(z)\\
   & \leq \frac 1d \int_a^b\int_a^b \|f(t)-f(y)\|_X^p\ud t\ud y \\
   & = \frac 1d\int_0^{d}\int_0^{d}\|f^{\ast}(t+y)-f^{\ast}(y)\|_X^p\ud t\ud y\\
   & = \frac 1d\int_0^{d}\int_0^{d}\|f^{\ast}(t+y)-f^{\ast}(y)\|_X^p\ud y\ud t\\
   &= \frac 1d \int_0^d I^{\ast}_{p,1}(t)\ud t,
\end{align*}
which shows that $a_0:=f(z)$ with $z\in S$ yields the assertion.
\end{proof}

The following lemma shows that we can bound integrals of lower order differences of periodic functions with integrals involving higher order differences.

\begin{lem}\label{lem-aux-2}
Let $0<p<\infty$ and $k\in \nat$.  Then we have the following relation
\[
\int_0^d I^{\ast}_{p,k}(y)\ud y\le c\, \int_0^d I^{\ast}_{p,k+1}(y)\ud y,
\]
with the constant $c>0$ only depending on $k$ and $p$, but otherwise independent of the function $f$ and the interval $[a,b)$.
\end{lem}

\begin{proof}
We make use of the following identity
\beq\label{est-09a}
\Delta^k_{2y}f^{\ast}(t)-2^k\Delta^k_y f^{\ast}(t)=\sum_{i=1}^k{k\choose i}\sum_{m=0}^{i-1}\Delta_y^{k+1}f^{\ast}(t+my),
\eeq
which can be found in \cite[Sect. 3.3.2]{Tim63}. Let  $0<p<1$. In this case we know that $|\cdot |^p$ is subadditive. This and  integration from $0$ to $d$ w.r.t. $t$ in \eqref{est-09a} leads to
\begin{multline*}
2^{kp}\int_0^d \|\Delta_y^k f^{\ast}(t)\|_X^p \ud t-\int_0^d \|\Delta^k_{2y}f^{\ast}(t)\|_X^p\ud t
\\
\leq \sum_{i=1}^k{k\choose i}\sum_{m=0}^{i-1}\int_0^d \|\Delta_y^{k+1}f^{\ast}(t+my)\|_X^p\ud t.
\end{multline*}
Now integrating once more from $0$ to $d$ w.r.t. $y$ and using the definition of $I^{\ast}_{p,k}$ gives
\beq \label{est-08b}
2^{kp}\int_0^d I^{\ast}_{p,k}(y)\ud y-\int_0^d I_{p,k}^*(2y)\ud y
\leq \sum_{i=1}^k{k\choose i}\, i \int_0^d I_{p,k+1}^{\ast}(y)\ud y .
\eeq
Since $I^{\ast}_{p,k}(y)$ is $d$-periodic, we have the identity
$$\int_0^d I^{\ast}_{p,k}(2y)\ud y=\frac 12\int_0^{2d} I^{\ast}_{p,k}(y')\ud y'=\int_0^d I^{\ast}_{p,k}(y)\ud y. $$
Inserting this in  \eqref{est-08b} we obtain
\beq \label{est-08a}
\left(2^{kp}-1\right)\int_0^d I^{\ast}_{p,k}(y)\ud y
\leq c_{k,p}\int_0^d I_{p,k+1}^{\ast}(y)\ud y,
\eeq
which gives the desired estimate in the case $0<p<1$.
When $1\leq p<\infty$ we proceed with \eqref{est-09a} as follows: We add $2^k\Delta^k_y f^{\ast}(t)$ on both sides of \eqref{est-09a} and integrate from $0$ to $d$ w.r.t. $t$ and from $0$ to $d$ w.r.t. $y$ afterwards in the $L_p$-norm. This gives \eqref{est-08b} but with the integrals to the power $\frac 1p$. We proceed as before and end up with  \eqref{est-08a} to the power $\frac 1p$, which  proves the asserted estimate.
\end{proof}

The following lemma shows how to bound integrals of higher order differences of the periodic extension of a function by integrals of higher order differences of the original function plus first order differences.

\begin{lem} \label{lem-aux-3}
Let $f\in L_p(I,X)$, where $0<p<\infty$. Then for any $k\in \nat$ it holds
\[
\int_0^{\frac dk}I^{\ast}_{p,k}(y)\ud y\leq 2\int_0^{\frac dk}I_{p,k}(y)\ud y
+ c \, d \, I_{p,1}\left(\frac dk\right).
\]
with the constant $c>0$ only depending on $k$ and $p$, but otherwise independent of the function $f$ and the interval $[a,b)$.
\end{lem}

\begin{proof}
By definition $\Delta^k_y f^{\ast}(t)=\sum_{i=0}^k(-1)^{k-i}{k\choose i}f^{\ast}(t+iy)$ and the fact that  $f^{\ast}$ is the $d$-periodic continuation of $f$, i.e.,  $f=f^{\ast}$ on $[a,b)$ and $f^{\ast}(t)=f(t-d)$ for some $t\in [b,b+d)$, we express $I_{p,k}^{\ast}$ in terms of the values of $f$ as follows:
\begin{align}
    I^{\ast}_{p,k}(y)
    &=\int_a^b \|\Delta_y^k f^{\ast}(t)\|_X^p\ud t \notag\\
    &= \int_a^{b-ky} \|\Delta_y^k f(t)\|_X^p\ud t+\sum_{j=1}^k \int_{b-jy}^{b-(j-1)y}\|S_j\|_X^p\ud t, \qquad 0 \le y\leq \frac dk, \label{est-09}
\end{align}
where
\[
S_j(t)=\sum_{i=0}^{j-1}(-1)^{k-i}{k\choose i}f(t+iy)+\sum_{i=j}^k(-1)^{k-i}{k\choose i}f(t+iy-d).
\]
\begin{minipage}{\textwidth}
\includegraphics[width=12cm]{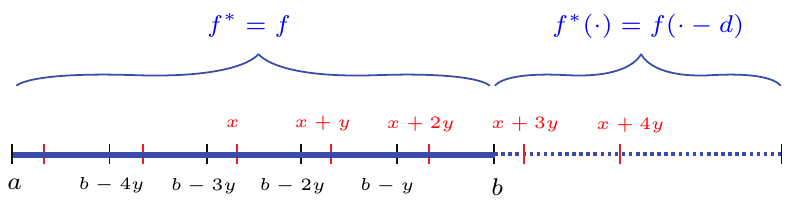}
\captionof{figure}{Express $f^{\ast}$ via $f$ with $j=3$}
\end{minipage}\\[0.2cm]
Now we transform $S_j$ as follows: we augment the first term of the first sum and the last of the second sum in order to obtain the value of the $k$-th difference of $f$ at the point $t$ with step $y-\frac dk$. This yields
\begin{align*}
    S_j&= \sum_{i=0}^k (-1)^{k-i}{k\choose i} \textcolor{blue}{f\left(t+i\Big(y-\frac dk\Big)\right) }
    \quad {\textcolor{Magenta}{\nearrow i=0 \text{ first sum \quad }}\atop \textcolor{Cyan}{\searrow i=k \text{ second sum}}} & (=:T_1)\\
    & \quad + \textcolor{Magenta}{\sum_{i=1}^{j-1}} (-1)^{k-i}{k\choose i} \left[f\left(t+iy\right)\textcolor{blue}{-f\left(t+i\Big(y-\frac dk\Big)\right)}\right]& (=:T_2(j))\\
    & \quad + \textcolor{Cyan}{\sum_{i=j}^{k-1}} (-1)^{k-i}{k\choose i} \left[f\left(t+iy-d\right)\textcolor{blue}{-f\left(t+i\Big(y-\frac dk\Big)\right)}\right] & (=:T_3(j))\\
    &=: T_1+T_2(j)+T_3(j).
\end{align*}
Since $T_1$ does not depend on $j$,
\begin{align}
\sum_{j=1}^k \int_{b-jy}^{b-(j-1)y}\|T_1\|_X^p \,\ud t
&=\int_{b-ky}^{b}\|T_1\|_X^p \ud t\notag
 =\int_{b-ky}^{b}\|\Delta^k_{y-\frac dk}f(t)\|_X^p \,\ud t\notag \\
%&=\int_{a}^{b-k(\frac dk-y)}\|\Delta^k_{\frac dk-y}f(t)\|_X^p \,\ud t\notag \\
&=\int_{a}^{a+ky}\|\Delta^k_{\frac dk-y}f(t)\|_X^p \ud t
=I_{p,k}\left(\frac dk-y\right),   \label{est-10}
\end{align}
where in the third step we changed the step $y-\frac dk$ involving the $k$-th difference of $f$ into $\frac dk-y$ in order to obtain a nonnegative step. We now estimate the sum
\begin{align}
    \sum_{j=1}^k &\int_{b-jy}^{b-(j-1)y}\|T_2(j)\|_X^p+\|T_3(j)\|_X^p \ud t\notag\\
    &\lesssim \sum_{j=1}^k\int_{b-jy}^{b-(j-1)y}  \left\{
    \sum_{i=1}^{j-1}{k\choose i}^p \Big\|f(t+iy)-f\left(t+i\Big(y-\frac dk\Big)\right)\Big\|_X^p\right.\notag\\
    & \qquad \qquad \qquad\qquad
    +  \left.\sum_{i=j}^{k-1}{k\choose i}^p \Big\|f(t+iy-d)-f\left(t+i\Big(y-\frac dk\Big)\right)\Big\|_X^p  \right\}\ud t\notag\\
    \intertext{\small \qquad \qquad   (change summation $\sum_{j=1}^k \sum_{i=1}^{j-1}=\sum_{i=1}^{k-1} \sum_{j=i+1}^{k}$
      and $ \sum_{j=1}^k \sum_{i=j}^{k-1}=\sum_{i=1}^{k-1} \sum_{j=1}^{i}$)}
    &=
    \sum_{i=1}^{k-1}{k\choose i}^p \int_{b-ky}^{b-iy}\left\|f(t+iy)-f\left(t+i\Big(y-\frac dk\Big)\right)\right\|_X^p\ud t\notag\\
     & \qquad \qquad \qquad\qquad \qquad
    +  \sum_{i=1}^{k-1}{k\choose i}^p \int_{b-iy}^b\left\|f(t+iy-d)-f\left(t+i\Big(y-\frac dk\Big)\right)\right\|_X^p \ud t \notag\\
    \intertext{\small \qquad \qquad \qquad (1st integral: Substitution $t'=t+i\left(y-\frac dk\right)$ ; reverse sum $i\mapsto k-i$)}
    \intertext{\small \qquad \qquad \qquad (2nd integral: Substitution $t''=t+iy-d$)}
     &=
    \sum_{i=1}^{k-1}{k\choose i}^p \int_{\frac{k-i}{k}a+\frac ikb-iy}^{\frac{k-i}{k}a+\frac ik b}\left\|f\left(t'+(k-i)\frac dk\right)-f\left(t'\right)\right\|_X^p\ud t'\notag\\
     & \qquad \qquad \qquad\qquad \qquad
    +  \sum_{i=1}^{k-1}{k\choose i}^p \int_{a}^{a+iy}\left\|f(t'')-f\left(t''+(k-i)\frac{d}{k}\right)\right\|_X^p \ud t''. \label{est-08}
\end{align}
Now \eqref{est-09}, \eqref{est-10}, and \eqref{est-08} yield
\begin{align*}
    I^{\ast}_{p,k}(y)
    \leq{}& I_{p,k}(y)+I_{p,k}\left(\frac dk-y\right) \\
    &   +\sum_{i=1}^{k-1}{k\choose i}^p \int_{\frac{k-i}{k}a+\frac ikb-iy}^{\frac{k-i}{k}a+\frac ik b}\left\|f\left(t'+(k-i)\frac dk\right)-f\left(t'\right)\right\|_X^p\ud t'\notag\\
     &
    +  \sum_{i=1}^{k-1}{k\choose i}^p \int_{a}^{a+iy}\left\|f(t'')-f\left(t''+(k-i)\frac{d}{k}\right)\right\|_X^p \ud t''.
\end{align*}
Integrating from $0$ to $\frac dk$ w.r.t. $y$ gives
\begin{align*}
    \int_0^{\frac dk}I^{\ast}_{p,k}(y)\ud y
    \le{}&  \int_0^{\frac dk}I_{p,k}(y)\ud y+ \int_0^{\frac dk}I_{p,k}\left(\frac dk-y\right)\ud y \\
    &  +c \left\{ \sum_{i=1}^{k-1}\int_0^{\frac dk} \int_{\frac{k-i}{k}a+\frac ikb-iy}^{\frac{k-i}{k}a+\frac ik b}\left\|f\left(t'+(k-i)\frac dk\right)-f\left(t'\right)\right\|_X^p\ud t' \ud y\right.\notag\\
     & \quad
    +  \left.\sum_{i=1}^{k-1}\int_0^{\frac dk} \int_{a}^{a+iy}\left\|f(t'')-f\left(t''+(k-i)\frac{d}{k}\right)\right\|_X^p \ud t''\ud y\right\}.
\end{align*}
\begin{minipage}{0.6\textwidth} We change the order of integration in the double integrals. For the second integral this yields
\[
\int_0^{\frac dk}\int_a^{a+iy}(\ldots)\ud t''\ud y
\longrightarrow \int_a^{a+i\frac dk}\int_{\frac{t''-a}{i}}^{\frac dk}(\ldots)\ud y \ud t''.
\]
Similarly for the first one. Moreover,
 observing that the integrand in both cases does not depend on $y$ we obtain
\end{minipage}\hfill \begin{minipage}{0.35\textwidth}
\includegraphics[width=6cm]{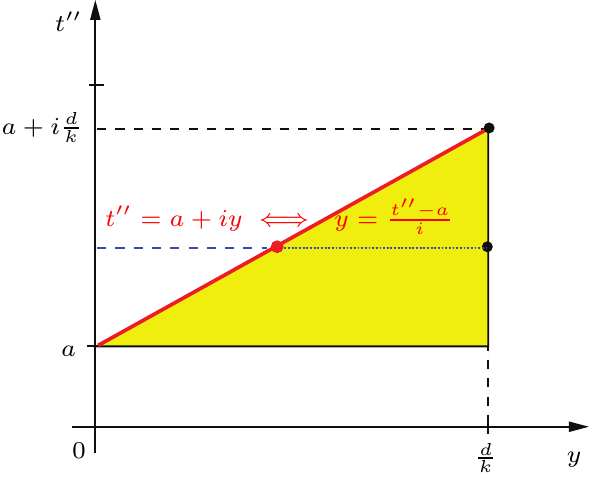}
\end{minipage}\\
\begin{align}
    \int_0^{\frac dk}&I^{\ast}_{p,k}(y)\ud y\notag\\
    &\le  2\int_0^{\frac dk}I_{p,k}(y)\ud y
     + c\bigg\{ \sum_{i=1}^{k-1} \int_a^{a+i\frac dk}\left(\frac{t'}{i}-\frac ai\right)\left\|f\left(t'+(k-i)\frac dk\right)-f\left(t'\right)\right\|_X^p\ud t' \notag\\
     & \qquad
    +  \sum_{i=1}^{k-1}\int_a^{a+i\frac dk}\left(\frac dk-\frac{t''-a}{i}\right)\left\|f(t'')-f\left(t''+(k-i)\frac{d}{k}\right)\right\|_X^p \ud t'' \bigg\} \notag\\
    &=  2\int_0^{\frac dk}I_{p,k}(y)\ud y
    + c \frac dk\sum_{i=1}^{k-1} \int_a^{a+i\frac dk}\left\|f(t)-f\left(t+(k-i)\frac{d}{k}\right)\right\|_X^p \ud t.\label{est-11}
\end{align}
Using a telescopic sum we see that
\begin{align*}
\Big\|f(t)-f\left(t+(k-i)\frac dk\right)\Big\|^p_X
\lesssim \sum_{j=1}^{k-i}\Big\|f\left(t+(j-1)\frac dk\right)-f\left(t+j \frac dk\right)\Big\|^p_X
%\intertext{\hfill (telescope sum)}
\end{align*}
and for $i=1,2,\ldots, k-1$,
\begin{align}
   \sum_{j=1}^{k-i}&\int_a^{a+i\frac dk}\left\|f\left(t+(j-1)\frac dk\right)-f\left(t+j\frac dk\right)\right\|_X^p\ud t\notag\\
    &=\sum_{j=1}^{k-i}\int_{a+\frac{j-1}{k}d}^{a+\frac {i+j-1}{k}d}\left\|f\left(t'\right)-f\left(t'+\frac dk\right)\right\|_X^p\ud t'\notag\\
    &\leq (k-i) \int_{a}^{a+\frac{k-1}{k}d}\left\|f\left(t'\right)-f\left(t'+\frac dk\right)\right\|_X^p\ud t', \label{est-12}
\end{align}
where in the second step we used a change of variables $t':=t+(j-1)\frac{d}{k}$. Inserting \eqref{est-12} into \eqref{est-11} finally gives
\begin{align*}
    \int_0^{\frac dk}I^{\ast}_{p,k}(y)\ud y
    &\le 2\int_0^{\frac dk}I_{p,k}(y)\ud y
    + d \, c_{k,p} \int_{a}^{b-\frac dk}\left\|f\left(t'\right)-f\left(t'+\frac dk\right)\right\|_X^p\ud t'\\
    &=2\int_0^{\frac dk} I_{p,k}(y)\ud y
    + c_{k,p}\, d \, I_{p,1}\left(\frac dk\right),
\end{align*}
which completes the proof.
\end{proof}

The previous lemmas give the following result, which shows that we can bound the best  approximation of a  function $f\in L_p(I,X)$ by %with 
a constant $a_0\in X$ with the help of integrals of higher order differences and first order differences of $f$.

\begin{lem}\label{aux-lem-OS78}
Let $I=[a,b)$, $0<p<\infty$, and $m\in \nat$. There exists a constant $c=c_{m,p}$, such that for every $f\in L_p(I,X)$ there exists $a_0\in X$ satisfying, for $h=\frac{b-a}{m}$,
\begin{align}
\|f-a_0\|^p_{L_p(I,X)}
&\leq c \left[\frac 1h \int_0^h \int_a^{b-ms}\|\Delta_s^mf(t)\|_X^p\ud t\ud s+ \int_a^{b-h}\|\Delta_h f(t)\|_X^p\ud t\right]\label{aux-jackson}\\
%&=c \left[\frac 1h \|\Delta_s^mf(t)\|^p_{L_p(I_{mh},X)}\ud s+ \|\Delta_h f(t)\|^p_{L_p(I_h,X)}\right]\notag\\
&= c \left[w_m(f,I,h)_p^p+ \|\Delta_h f(t)\|^p_{L_p(I_h,X)}\right].\notag
\end{align}
\end{lem}

\begin{rem}
Note that the second term with the first order differences in \eqref{aux-jackson} is crucial: If $f$ is a polynomial of degree $m-1$ the first integral on the right-hand side vanishes but the left-hand side might not.
\end{rem}

\begin{proof}
We first notice that, by induction, we can easily check that
\[
\Delta_{my}^m f^*(t) = \sum_{i_m=0}^{m-1} \dots \sum_{i_1=0}^{m-1} \Delta_y^m f^*(t+i_1y+\dots+i_my),
\]
so that $I_{p,m}^*(my) \lesssim I_{p,m}^*(y)$.
Taking $a_0\in X$ as constructed in  Lemma \ref{lem-aux-1} and using   Lemmas \ref{lem-aux-2} and \ref{lem-aux-3} with $k=m$, setting $h = \frac dm$, we obtain
\begin{align*}
     \|f-a_0\|_{L_p(I,X)}^p
    &\ \lesssim \frac 1d \int_0^d I^{\ast}_{p,1}(y)\ud y
    \lesssim \frac 1d \int_0^d I^{\ast}_{p,m}(y)\ud y \\
    &\ \lesssim \frac1d \int_0^d I_{p,m}^*\left(\frac ym\right) \ud y
    =\frac md \int_0^{\frac dm}I^{\ast}_{p,m}(y')\ud y'\\
    &\ \leq \frac md\left[2\int_0^{\frac dm}I_{p,m}(y)\ud y+c_{p,m}dI_{p,1}\left(\frac dm\right)\right]\\
    &\ =c'_{p,m} \frac 1h \left[ \int_0^{h}\int_a^{b-my}\|\Delta_y^m f \|_X^p\ud t\ud y+\int_a^{b-h}\|\Delta_h f(t)\|_X^p\ud t\right],
\end{align*}
which is the desired result.
\end{proof}

Finally, a repeated application of
%Repeatedly  applying   
Lemma \ref{aux-lem-OS78} 
now allows us to establish Jackson's inequality.

\begin{proof}[Proof of Theorem~\ref{thm-jackson}]
We assume $I = [0,1)$. The general case follows by scaling, using~\eqref{scaling}.
Let $h=\frac{1}{2r}$, $f\in L_p(I,X)$, and  denote the approximant $a_0\in X$ from Lemma \ref{aux-lem-OS78} by $M(f,I):=a_0$. Now
define the coefficients $a_0, \ldots, a_{r-1}$ recursively as follows:
\begin{align*}
    a_{r-1}&=M\left(\Delta^{r-1}_hf, [0,1-(r-1)h]\right)\frac{1}{h^{r-1}}\frac{1}{(r-1)!},\\
    f_1(t)&=f(t)-a_{r-1}t^{r-1}, \\
    a_{r-2}&=M(\Delta^{r-2}_hf_1, [0,1-(r-2)h])\frac{1}{h^{r-2}}\frac{1}{(r-2)!},\\
    f_2(t)&=f_1(t)-a_{r-2}t^{r-2}=f(t)-(a_{r-1}t^{r-1}+a_{r-2}t^{r-2}), \\
    \vdots & \\
    a_{2}&=M(\Delta^{2}_hf_{r-3}, [0,1-2h])\frac{1}{h^{2}}\frac{1}{2!},\\
    f_{r-2}(t)&=f_{r-3}(t)-a_{2}t^{2}=f(t)-(a_{r-1}t^{r-1}+a_{r-2}t^{r-2}+\ldots + a_2 t^2), \\
    a_{1}&=M(\Delta^{1}_hf_{r-2}, [0,1-h])\frac{1}{h},\\
     f_{r-1}(t)&=f_{r-2}(t)-a_{1}t=f(t)-(a_{r-1}t^{r-1}+a_{r-2}t^{r-2}+\ldots + a_1 t), \\
     a_0&=M(f_{r-1},[0,1]).
\end{align*}
With
\[
P_r(t)=\sum_{k=0}^{r-1}a_kt^k=a_0+a_1t+\ldots +a_{r-1}t^{r-1}
\]
we compute
\begin{align*}
    \|f-P_r\|^p_{L_{p}(I,X)}&=\|f_{r-1}-a_0\|^p_{L_p(I,X)}
    =\|f_{r-1}-M(f_{r-1},[0,1])\|^p_{L_p(I,X)}\\
    &\lesssim w_{2r}\left(f_{r-1},I,\frac{1}{2r}\right)_p^p+\|\Delta_h f_{r-1}\|_{L_p(I_h,X)}^p
    \intertext{\hfill (\text{which follows from applying Lem. \ref{aux-lem-OS78} with } $m=2r$)}
    &=w_{2r}\left(f,I,h\right)_p^p+\|\Delta_h f_{r-2}-a_1h\|_{L_p(I_h,X)}^p \\
    &=w_{2r}\left(f,I,h\right)_p^p+\|\Delta_h f_{r-2}-M(\Delta_h f_{r-2},[0,1-h])\|_{L_p(I_h,X)}^p \\
    & \lesssim w_{2r}\left(f,I,h\right)_p^p
    +w_{2r-1}\left(\Delta_h f_{r-2},[0,1-h],h\right)_p^p +\|\Delta^2_h f_{r-2}\|^p_{L_p(I_{2h},X)} \intertext{\hfill  (which follows from applying Lem. \ref{aux-lem-OS78} with $m=2r-1$)}
    & \lesssim
    w_{2r-1}\left(f,I,h\right)_p^p +\|\Delta^2_h f_{r-3}-a_2 2! h^2\|^p_{L_p(I_{2h},X)}
    \intertext{\hfill (we used~\eqref{inductivebound})}\\
    & =
    w_{2r-1}\left(f,I,h\right)_p^p +\|\Delta^2_h f_{r-3}-M(\Delta^2_hf_{r-3}, [0,1-2h])\|^p_{L_p(I_{2h},X)}\\
     & \lesssim
    w_{2r-2}\left(f,I,h\right)_p^p +\|\Delta^3_h f_{r-4}-a_3 3! h^3\|^p_{L_p(I_{3h},X)}
    \intertext{\hfill  (\text{which follows from applying Lem. \ref{aux-lem-OS78} with } $m=2r-2$)}
    & \qquad \vdots \\
     & \lesssim
    w_{r+2}\left(f,I,h\right)_p^p +\|\Delta^{r-1}_h f-a_{r-1} (r-1)! h^{r-1}\|^p_{L_p(I_{(r-1)h},X)}\\
    & \lesssim
    w_{r+1}\left(f,I,h\right)_p^p +\|\Delta^{r}_h f\|^p_{L_p(I_{rh},X)} \intertext{ \hfill (\text{which follows from applying Lem. \ref{aux-lem-OS78} with } $m=r+1$)}
    & \leq  w_{r+1}\left(f,I,h\right)_p^p+w_r(f,I,h)_p^p\\
     & \lesssim  w_{r}\left(f,I,h\right)_p^p,
\end{align*}
which proves the theorem.
\end{proof}

\begin{rem}
Note that Theorem \ref{thm-jackson} also holds for $p=\infty$: if we extend  \eqref{I_p_1} by
\[
I^{\ast}_{\infty,k}(h):=\sup_{t\in  [a,b)}\|\Delta^k_hf^{\ast}(t)\|_X=\sup_{t\in  [0,d) }\|\Delta^k_hf^{\ast}(t)\|_X,
\]
and similarly \eqref{I_p_2}, then Lemmas \ref{lem-aux-1}--\ref{aux-lem-OS78} can be extended to the case when $p=\infty$ by obvious modifications in the proofs, i.e., mostly replacing the integrals by  suprema.  In this case   the factors `$\frac 1d$' and `$d$' in Lemmas  \ref{lem-aux-1} and \ref{lem-aux-3}, respectively,  disappear.
\end{rem}

\subsection{Whitney's estimate}\label{S:Whitney}

Having established Jackson's estimate \eqref{Jackson} in Theorem~\ref{thm-jackson} we now proceed to prove Whitney's estimate.

\begin{thm}[Generalized Whitney's theorem]\label{thm-gen-whitney}
Let $0 < p , q \le \infty$, $r\in \nat$, and $s>0$.
	If  $\left( 1/q-1/p\right)_+ \le s < r$ then there exists a constant $c >0$ which depends only on $p$, $q$, $r$ such that
\begin{equation}\label{whitney}
E_r(f,I)_p=\inf_{P\in \mathbb{V}^r_{I,X}}\|f-P\|_{L_p(I, X)}\le c |I|^{s+\frac 1p-\frac 1q}|f|_{B^s_{q,q}(I,X)},
\end{equation}
for all $f \in B^s_{q,q}(I,X)$ and for any finite interval $I$.
\end{thm}

Since this involves the $L_p$-norm on the left-hand side and an $L_q$-norm on the right-hand side, we first deal with the problem of how to switch from $p$-norms to $q$-norms for vector-valued polynomials. Using this together with the Jackson estimate, the fact that  according to Lemma \ref{discrete-seminorm} we can express the quasi-norm of the Besov spaces $B^s_{p,q}(I,X)$ as a discrete summation instead of integrals yields  Whitney's estimate.

%In order to establish \eqref{scaling_lpq} we make use of the following result.

\begin{lem}\label{lem-scaling-aux}
Let $0<p<\infty$ and $I=[0,1]$. On $\V^r_{I,X}$ the quasi-norm
\[
\|P\|_p:=\left(\int_0^1 \|P(t)\|_X^p\ud t\right)^{1/p}
\]
is equivalent to the norm
\[
\|P\|_{\ast}:=\max_{j=1,\ldots, r}\|P_j\|_X,
\qquad P_j = P(t_j), \quad t_j = \frac{j-1}{r-1}, \quad j=1,\dots,r.
\]
The constants involved in the equivalence depend on $r$ and $p$, but are otherwise independent of $P \in \V_{I,X}^r$.
\end{lem}

\begin{rem}
At first sight, it may seem that this lemma is obvious, because it looks like an equivalence of quasi-norms in a finite-dimensional space. But this is not the case, since the space $\V^r_{I,X}$ is not finite-dimensional, when $X$ is an arbitrary Banach space. \\
With slight modifications in the proof, Lemma \ref{lem-scaling-aux} also holds for $p=\infty$ and the quasi-norm
$
\|P\|_{\infty}=\sup_{t\in  [0,1]}\|P(t)\|_X.
$
\end{rem}

\begin{proof}
Let $\{ \ell_j \}_{j=1}^r$ denote the Lagrange basis of $\Pi^r$ corresponding to the equally spaced nodes $t_j = \frac{j-1}{r-1}$, $j=1,\dots,r$ on $[0,1]$, i.e.,
\[
\ell_j(t)=\prod_{i\neq j}\frac{t-t_i}{t_j-t_i}
%\ \in \  \mathcal{P}_{r}
, \quad \text{so that}\quad  \ell_j(t_i) = \delta_{ij}
\quad  \text{and}\quad
P = \sum_{j=1}^r P_j \ell_j, \text{ if $P \in \V_{I,X}^r$}.
\]
Obviously, for $P \in \V_{I,X}^r$,
\begin{align*}
    \|P\|_p&=\left(\int_0^1 \|P(t)\|_X^p\ud t\right)^{1/p}=\left(\int_0^1 \Big\|\sum_{j=1}^{r} \ell_j(t)P_j\Big\|_X^p\ud t\right)^{1/p}\\
    &\leq c_{p,r} \sum_{j=1}^{r}\left(\int_0^1 \ell_j^p(t) \|P_j\|_X^p\ud t\right)^{1/p}
     \leq c_{p,r}  \max_{j=1,\ldots, r}\|P_j\|_X = c_{p,r}  \|P\|_{\ast}.
\end{align*}
Let now $P = \sum_{j=1}^r P_j \ell_j \in \V_{I,X}^r$ and let $i$ be such that $\|P_i\|_X=\max_j \|P_j\|_X=\|P\|_{\ast}$. Then, for each $t \in I$, we have
\begin{align}
    \|P(t)\|_X
    &= \Big\|\sum_{j=1}^{r}\ell_j(t)P_j\Big\|_X \notag
    \geq |\ell_i(t)|\|P_i\|_X-\sum_{j\neq i}|\ell_j(t)|\|P_j\|_X
    \\
    &\geq \|P\|_{\ast}\bigg(|\ell_i(t)|-\sum_{j\neq i}|\ell_j(t)|\bigg). \label{est-lj}
\end{align}
Since at the point $t=t_i$ we have $\ell_i(t_i)=1$ and $\ell_j(t_i)=0$ for all $j\neq i$, there exists $\delta>0$ such that
\[
|t-t_i|<\delta \quad \Longrightarrow \quad |\ell_i(t)|>\frac 34>\frac 14>\sum_{j\neq i}|\ell_j(t)|;
\]
notice that $\delta > 0$ can be chosen independent of $i$, but will depend on $r$.
Hence,
\[
|\ell_i(t)|-\sum_{i\neq j}|\ell_j(t)|>\frac 12.
\]
Hence, \eqref{est-lj} gives us
\[
\|P(t)\|_X\geq \frac 12 \|P\|_{\ast}\qquad \text{for}\qquad |t-t_i|<\delta.
\]
Raising to the power $p$ and averaging over the interval $(t_i-\delta, t_i+\delta)\cap I$ yields
\[
\|P\|_{\ast}\leq  \left(\frac{2^{p}}{\delta}\int_{(t_i-\delta, t_i+\delta)\cap I}\|P(t)\|^p_X \ud t\right)^{1/p}\le \bar c_{p,r} \left(\int_{I}\|P(t)\|^p_X \ud t\right)^{1/p}= \bar c_{p,r}\|P\|_{p},
\]
and the assertion follows.
\end{proof}

By a scaling argument we obtain from the previous Lemma the following equivalence of $L_p(I,X)$ norms in $\V_{I,X}^r$ on an arbitrary interval $I$. The proof is very simple and is thus omitted.

\begin{cor}\label{cor-scaling-lpq}
Let $0<p,q\leq\infty$ and $r \in \N$. Then there exists a constant $c >0$ which depends only on $p$, $q$, $r$ such that on any finite interval $I$,
	\beq\label{scaling_lpq}
	\|P\|_{L_p(I,X)} \leq c |I|^{1/p-1/q}\|P\|_{L_q(I,X)},
	\qquad \forall P\in \mathbb{V}^r_{I,X}.
	\eeq
%	for any finite interval $I$.
\end{cor}

\begin{comment}
\begin{proof} By Lemma \ref{lem-scaling-aux}, all norms $\|\cdot \|_{L_p([0,,1))}$, $0<p<\infty$, are equivalent to $\|\cdot\|_{\ast}$ in $\V_{[0,1),X}$. On an arbitrary interval $I=[a,b]$ using the substitution $y:=|I|^{-1}t$ we have that, for $P \in \V_{I,X}$,
    \begin{align*}
        \|P\|_{L_p(I,X)}&= \left(\int_{{I}} \|P(t)\|_X^p\ud t\right)^{1/p}\\
        &= |{I}|^{1/p}\left(\int_{[0,1]} \|P(y)\|_X^p\ud y\right)^{1/p}
        \simeq |I|^{1/p}\left(\int_{[0,1]} \|P(y)\|_X^q\ud y\right)^{1/q}\\
        &=|{I}|^{1/p-1/q}\left(\int_{{I}} \|P(t)\|_X^q\ud t\right)^{1/q}=|{I}|^{1/p-1/q} \|P\|_{L_q(I,X)},\\
    \end{align*}
    which gives the desired result.
\end{proof}
\end{comment}

Following the steps from~\cite[Sec.~6.1]{DeV98} we can now prove Whitney's estimate in $B^s_{q,q}(I,X)\cap L_p(I,X)$.

\begin{prop}\label{prop-gen-whitney}
Let $0 < p , q \le \infty$, $r\in \nat$, and $s>0$.
	If  $\left( 1/q-1/p\right)_+ \le s < r$ then there exists a constant $c >0$ which depends only on $p$, $q$, $r$ such that
\begin{equation}\label{whitney0}
E_r(f,I)_p:=\inf_{P\in \mathbb{V}^r_{I,X}}\|f-P\|_{L_p(I, X)}\le c |I|^{s+\frac 1p-\frac 1q}|f|_{B^s_{q,q}(I,X)},
\end{equation}
for all $f \in B^s_{q,q}(I,X)\cap L_p(I,X)$ and for any finite interval $I$.
\end{prop}

\begin{proof}
Since $E_{r+1}(f,I)_p \le E_r(f,I)_p$, it is sufficient to prove the result in the case $r = \lfloor s \rfloor + 1$, and by scaling it is sufficient to consider  $I = [0,1)$.
Also, since $E_r(f,I)_p \le E_r(f,I)_q$ when $p < q$, it is sufficient to consider the case $q\le p$.

Let $D_k$ for $k=0,1,2,\ldots$ denote the following dyadic partitions of $I$:
\[
D_k:=\{
I_k^j:=2^{-k}[j-1,j), \ j=1,\ldots, 2^{k}
\}.
\]
We let $S_k$ denote a piecewise polynomial function of order $r$ on the partition $D_k$ satisfying the Jackson estimate \eqref{Jackson} with $p$ replaced by $q$, in each sub-interval, i.e.,
\[
\|f-S_k\|_{L_q(I^j_k,X)}\lesssim w_r(f,I_k^j,2^{-k})_q,
\qquad j=1,2,\dots,2^k, \quad k=0,1,\dots,
\]
whence $S_0 \in \V_{I,X}^r$.

Then, on the one hand, we have
\begin{align*}
    \|f-S_k\|^q_{L_q(I,X)}
    &=\sum_{j=1}^{2^k}\|f-S_k\|^q_{L_q(I^j_k,X)}\lesssim \sum_{j=1}^{2^k}w_r(f,I_k^j,2^{-k})_q^q.
\end{align*}
Denoting $\tilde{I}^j_k = \big( I_k^j \big)_{rh}$ we obtain
\begin{align}
    \|f-S_k\|^q_{L_q(I,X)}
    &\lesssim \frac{1}{2^{-k}}\int_0^{2^{-k}}\sum_{j=1}^{2^k} \|\Delta_h^r f\|^q_{L_q(\tilde{I}^j_k,X)}\ud h\notag\\
    &=\frac{1}{2^{-k}}\int_0^{2^{-k}} \sum_{j=1}^{2^k}\int_{\tilde{I}^j_k} \|\Delta_h^r f(t)\|_X^q \ud t\ud h\notag\\
    & \leq \frac{1}{2^{-k}}\int_0^{2^{-k}} \int_{[0,1-rh]}\|\Delta_h^r f(t)\|_X^q \ud t\ud h\notag\\
    &= \frac{1}{2^{-k}}\int_0^{2^{-k}} \|\Delta_h^r f\|^q_{L_q([0,1-rh],X)}\ud h %\notag\\
    =w_r(f,I, 2^{-k})_q^q. \label{est-03}
\end{align}

On the other hand, using \eqref{scaling_lpq} in each subinterval $I_{k+1}^j$, we have
\begin{align}
    \|S_k-S_{k+1}\|^p_{L_p(I,X)}
    &=\sum_{j=1}^{2^{k+1}} \|S_k-S_{k+1}\|^p_{L_p(I_{k+1}^j,X)}\notag\\
    % &\lesssim 2^{-(k+1)\left(1-\frac pq\right)}\sum_{j=1}^{2^{k+1}}\|S_k-S_{k+1}\|^p_{L_q(I_{k+1}^j,X)}\notag\\
    &\lesssim 2^{-k\left(1-\frac pq\right)}\sum_{j=1}^{2^{k+1}}\|S_k-S_{k+1}\|^p_{L_q(I_{k+1}^j,X)}\notag\\
    &\lesssim 2^{-k\left(1-\frac pq\right)}\left(\sum_{j=1}^{2^{k+1}}\|S_k-S_{k+1}\|^q_{L_q(I_{k+1}^j,X)}\right)^{p/q}\notag\\
    &=2^{-k\left(1-\frac pq\right)}\|S_k-S_{k+1}\|^p_{L_q(I,X)},\label{est-05}
\end{align}
where
%in the second line we can use \eqref{scaling_lpq} since $S_{k}-S_{k+1}$ is a polynomial on $I^j_{k+1}$ and
in the second to last line we used the fact that $\ell_{q/p}\hookrightarrow \ell_1$ for $q\le p$. This yields for  $\overline{p}=\min\{1,p\}$,
\beq\label{est_02}
  \|S_k-S_{k+1}\|^{\overline{p}}_{L_p(I,X)}\lesssim 2^{-k\left(\frac 1p-\frac 1q\right)\overline{p}}\|S_k-S_{k+1}\|^{\overline{p}}_{L_q(I,X)}.
\eeq
But then using \eqref{est-03}, \eqref{est_02}, and the assumption that $f \in L_p(I,X)$, we obtain
\begin{align}
    E_r(f,I)_p^{\overline{p}}
    &\le \|f-S_0\|^{\overline{p}}_{L_p(I,X)} %\notag\\
    \leq \sum_{k=0}^{\infty} \|S_k-S_{k+1}\|^{\overline{p}}_{L_p(I,X)}\notag\\
    &\lesssim \sum_{k=0}^{\infty}2^{-k\left(\frac 1p-\frac 1q\right)\overline{p}}\|S_k-S_{k+1}\|^{\overline{p}}_{L_q(I,X)}\notag\\
    &\leq \sum_{k=0}^{\infty}2^{-k\left(\frac 1p-\frac 1q\right)\overline{p}}\left(\|S_k-f\|^{\overline{p}}_{L_q(I,X)}+\|f-S_{k+1}\|^{\overline{p}}_{L_q(I,X)}\right)\notag\\
    &\lesssim \sum_{k=0}^{\infty}2^{-k\left(\frac 1p-\frac 1q\right)\overline{p}}\|f-S_k\|^{\overline{p}}_{L_q(I,X)}\notag\\
    &\lesssim \sum_{k=0}^{\infty}2^{-k\left(\frac 1p-\frac 1q\right)\overline{p}}w_r(f,I,2^{-k})_q^{\overline{p}}\notag\\
    &=  \sum_{k=0}^{\infty}2^{-k\left(\left(\frac 1p-\frac 1q\right)+s\right)\overline{p}}2^{{ks\overline{p}} }w_r(f,I,2^{-k})_q^{\overline{p}}\notag\\
    &\leq   \sum_{k=0}^{\infty}2^{-k\delta\overline{p}}2^{{ks\overline{p}} }w_r(f,I,2^{-k})_q^{\overline{p}}, \label{est-04}
\end{align}
where $\delta:=\left(\frac 1p-\frac 1q\right)+s\geq 0$ due to %by 
our  assumption $s\geq \left(1/q-1/p\right)_+$. %, i.e.,  $\left(\frac 1p-\frac 1q\right)+s>0$.
In \eqref{est-04} we proceed as follows: if $q<\overline{p}$ we make use of the embedding $\ell_q\hookrightarrow \ell_{\overline{p}}$ together with the fact that $2^{-k\delta \overline{p}}\leq 1$ and for $q>\overline{p}$ we apply  H\"older's inequality with $\frac{q}{\overline{p}}>1$. This  finally gives
\beq \label{est-final}
    \|f-S_0\|_{L_p(I,X)}
    \leq \left(\sum_{k=0}^{\infty}2^{{ksq} }w_r(f,2^{-k},I)_q^{q}\right)^{1/{q}}
    \simeq |f|_{B^s_{q,q}(I,X)}.
\eeq
The assertion thus follows by recalling that $S_0 \in \V_{I,X}^r$.
% If $q>p$ the calculations are simpler. In fact, estimate \eqref{est-05} is replaced by
% $\| S_k - S_{k+1} \|_{L_p(I,X)} \lesssim \| S_k - S_{k+1} \|_{L_q(I,X)}$, which thereby implies
% \beq \label{est-06}
%  \|f-S_0\|^{\overline{p}}_{L_p(I,X)}\lesssim  \sum_{k=0}^{\infty}2^{-ks\overline{p}}2^{{ks\overline{p}} }w_r(f,2^{-k},I)_p^{\overline{p}},
% \eeq
% instead of \eqref{est-04}. Since $q>p$ we have $L_q(I,X)\hookrightarrow L_p(I,X)$ since $I$ is bounded, therefore,  $w_r(f,2^{-k},I)_p\leq w_r(f,2^{-k},I)_q$. Moreover,  it also holds that  $q\geq \overline{p}$, thus, as before application of H\"older's inequality in \eqref{est-06} with $\frac{q}{\overline{p}}>1$  yields \eqref{est-final}.
\end{proof}

As a consequence of the previous theorem we have that under  the same  assumptions $B^s_{q,q}(I,X)$ is embedded into $L_p(I,X)$.

\begin{cor}\label{cor-gen-whitney}
Let $0<p,q \le \infty$, $r\in \nat$, and $s>0$.
	If  $\left( 1/q-1/p\right)_+ \le s$ then $B^s_{q,q}(I,X)$ is embedded into $L_p(I,X)$ and there exists a constant $c >0$ which depends only on $p$, $q$, $r$, and $s$ such that
\begin{equation*}%\label{whitney}
\|f\|_{L_p(I, X)}\le c \|f\|_{B^s_{q,q}(I,X)},
\end{equation*}
for all $f \in B^s_{q,q}(I,X)$ and for any finite interval $I$.
\end{cor}

\begin{proof}
Let $f \in B^s_{q,q}(I,X) \cap L_p(I,X)$,  $r \in \N$, $r > s$, and let $S_0$ be as in the proof of Theorem~\ref{thm-gen-whitney}. Then,
\[
\| f \|_{L_p(I,X)} \lesssim \| f - S_0 \|_{L_p(I,X)} +  \| S_0 \|_{L_p(I,X)}
\lesssim | f |_{B^s_{q,q}(I,X)} + \| S_0 \|_{L_q(I,X)}.
\]
Since $S_0 \in \V_{I,X}^r$ was chosen satisfying Jackson estimate~\eqref{Jackson} with $p$ replaced by $q$,
\[
\|S_0\|_{L_q(I,X)} \lesssim \|f-S_0\|_{L_q(I,X)} + \|f\|_{L_q(I,X)}
\lesssim w_r(f,I,1)_q + \|f\|_{L_q(I,X)} \lesssim \|f\|_{L_q(I,X)}.
\]
Therefore, for all $f \in B^s_{q,q}(I,X) \cap L_p(I,X)$,
\[
\| f \|_{L_p(I,X)} \lesssim | f |_{B^s_{q,q}(I,X)} +  \|f\|_{L_q(I,X)}
\lesssim \| f \|_{B^s_{q,q}(I,X)}. 
\]
Finally, since $B^s_{q,q}(I,X) \cap L_p(I,X)$ is dense in $B^s_{q,q}(I,X)$ the assertion follows.
\end{proof}

The generalized Whitney's theorem, Theorem~\ref{thm-gen-whitney},  is now a consequence of Proposition~\ref{prop-gen-whitney} and Corollary~\ref{cor-gen-whitney}.

%%%%%%%%%%%%%%%%%%%%%%%%%%%%%%%%%%%%%%%%%%%%%%%%%%%%%%%%%%%%%%%%%
\section{Adaptive approximation {in one variable}}\label{sec:onevariable}
%%%%%%%%%%%%%%%%%%%%%%%%%%%%%%%%%%%%%%%%%%%%%%%%%%%%%%%%%%%%%%%%%

\subsection{The stationary case}

Given a polyhedral space domain $\Omega\subset \real^n$, $n \in \N$,  we let $\TT(\calT_0)$ denote the set of all triangulations $\mathcal{T}$ (partitions into simplices) that are obtained by successive application of the bisection routine of \cite{Ste08} from a properly labeled initial triangulation $\mathcal{T}_0$ of $\Omega$.
%For each $\mathcal{T}\in \TT(\calT_0)$ we denote by $\#\mathcal{T}$ the number of elements of the partition and by $|\calT| = \#\calT - \#\calT_0$, i.e., the number of bisections needed to obtain $\calT$ from $\calT_0$.
If $n=1$, $\TT(\{0<T\})$ denotes the set of all partitions of $\Omega = [0,T)$ into sub-intervals that may be obtained by successive bisection of $\calT_0 = \{[0,T)\}$.
%These one-dimensional partitions may be required to satisfy that the length-ratio of neighboring sub-intervals is bounded above by $2$. This requirement is equivalent to asking that neighboring elements belong to consecutive generations, which is the same assumption that holds in the triangulations obtained by the algorithm from~\cite{Ste08}.
For simplicity, the one-dimensional partition $\{ [0=t_0,t_1), [t_1,t_2),\dots,[t_{N-1},t_N = T) \}$ will be usually denoted by $\{ 0=t_0 < t_1 < \dots < t_N = T \}$.
Whenever we write $\calT_* = \textsc{Refine($\calT,\calM$)}$, we understand that $\calM \subset \calT$ and $\calT_*$ is the refinement of $\calT$ obtained by the bisection routine of \cite{Ste08}.
In the one-dimensional case, we understand that $\calT_*$ is obtained by the sole replacement in $\calT$ of each element $T = [a,b) \in \calM$ by its children $[a,\frac{a+b}2)$, $[\frac{a+b}2,b)$.

Therefore, the following complexity bound holds:
\begin{quote}
Let $\calT_0$, $\calT_1$, $\calT_2$, \dots, be a sequence of partitions in $\TT(\calT_0)$ obtained by successive calls of $\calT_{k+1} = \textsc{Refine($\calT_k,\M_k$)}$, with $\M_k \subset \calT_k$ the set of \emph{marked} elements. Then, there exists a constant $C$ that depends on the initial triangulation $\calT_0$ such that
\begin{equation}
    \label{eq:complexity}
    \#\calT_k - \#\calT_0 \le C \sum_{j=0}^{k-1} \#\M_j, \qquad k=1,2,\dots.
\end{equation}
\end{quote}

For $\calT \in \TT(\calT_0)$, recall that $\mathbb{V}_{\mathcal{T}}^r$ is the finite element space of continuous piecewise polynomials of order $r$, i.e.,
\[
\mathbb{V}_{\mathcal{T}}^r :=\{g\in C(\overline{\Omega}): \ g\big|_{T}\in \Pi^r\ \text{for all } T\in \mathcal{T}\},
\]
where $\Pi^r$ denotes the set of polynomials of total degree (strictly) less than $r$. The underlying domain $\Omega$ and its dimension  are   implicitly indicated by the partition $\calT$, which will sometimes correspond to a time interval $[0,T)$ and sometimes to an $n$-dimensional space domain.\\

\paragraph{\bf Approximation Classes}
Let $X$ be a quasi-Banach  space  on the polyhedral %a 
bounded Lipschitz domain $\Omega \subset \R^n$ with  quasi-norm $\|\cdot\|_X$.
Let $\calT_0$ be a triangulation of $\Omega$, properly labeled so that~\eqref{eq:complexity} holds, and assume further that $\V_\calT^r \subset X$ for $\calT \in \TT(\calT_0)$. In this context, for $f\in X$, the best $N$-term approximation error is given by
\[
\sigma_N(f)=\inf_{|\mathcal{T}|\leq N}\inf_{g\in \mathbb{V}_{\mathcal{T}}^r}\|f-g\|_X.
\]
For $s>0$ we define the approximation class $\mathbb{A}_s(X)$ as the set of those functions in $X$ whose best $N$-term approximation error is of order $N^{-s}$, i.e.,
\[
\mathbb{A}_s(X):=\{f\in X: \ \exists c>0 \text{ such that } \sigma_N(f)\leq cN^{-s}, \ \forall N\in \nat\}.
\]
Equivalently, we can define $\mathbb{A}_s(X)$ through a semi-quasi-norm as follows:
\[
\mathbb{A}_s(X):=\{f\in X: \ |f|_{\mathbb{A}_s(X)}<\infty\}\quad \text{with}\quad |f|_{\mathbb{A}_s(X)}:=\sup_{N\in \nat}N^s\sigma_N(f).
\]
Alternatively, this definition is equivalent to saying that $f\in {\mathbb{A}_s(X)}$ if there is a constant $c$ such that for all $\varepsilon>0$, there exists a mesh $\mathcal{T}$ that satisfies
\begin{equation}\label{appr-class-1}
\inf_{g\in \V^{r}_{\mathcal{T}}}\|f-g\|_X\leq c \varepsilon \quad \text{and}\quad |\mathcal{T}|\leq \varepsilon^{-1/s},
\end{equation}
and $|f|_{\mathbb{A}_s(X)}$ is equivalent to the infimum of all constants $c$ that satisfy \eqref{appr-class-1}.

We use the following result from \cite[Thm. 2.2, Cor. 2.3]{GM14}, which is the  high-order analog to the one presented in~\cite{BDDP02} for linear finite elements ($r=2$).

\begin{thm}\label{space-poly}
Let $X=B^{\alpha}_{p,p}(\Omega)$, $0<p<\infty$, $0<\alpha<\min\{r,1+\frac 1p\}$ or $X=L_p(\Omega)$ if $\alpha=0$. If $f\in B^{s+\alpha}_{\tau,\tau}(\Omega)$  with $s>0$, $0<\frac{1}{\tau}<\frac sn+\frac 1p$,  and
$s+\alpha < r$,
%$s+\alpha\leq r+\frac{1}{\min\{1,\tau\}}$,
then
\begin{align}
B^{\alpha+s}_{\tau,\tau}(\Omega)&\subset \mathbb{A}_{s/n}(B^{\alpha}_{p,p}(\Omega))\qquad (\alpha>0), \label{approx-class-2}\\
B^{s}_{\tau,\tau}(\Omega)&\subset \mathbb{A}_{s/n}(L_{p}(\Omega))\qquad \ \;  (\alpha=0).   \label{approx-class-3}
\end{align}
\end{thm}

In particular, if $p=2$ and $\alpha = 0$ we have the following result.

\begin{cor}\label{cor:spaceapprox}
Let $X = L_2(\Omega)$, $r \in \N$, $0<s < r$, and  $0<\frac 1\tau < \frac sn + \frac 12$. Then there exists a constant $C = C(r,s,\tau,\Omega,\TT)$ such that, for every $\varepsilon > 0$ there exists $\calT \in \TT(\calT_0)$ and $g \in \V^{r}_\calT$ such that
\[
\| f - g \|_X \le \varepsilon\, |f|_{B^s_{\tau,\tau}(\Omega)}   \qquad\text{and}\qquad
|\calT| \lesssim \varepsilon^{-n/s}.
\]
\end{cor}

\subsection{Greedy algorithm} % for polynomials in time}

Theorem~\ref{space-poly}, or equivalently Corollary~\ref{cor:spaceapprox}, is proved with the help of a so called \emph{Greedy algorithm}. In order to make this article self-contained, we present it here and use it to build a quasi-optimal partition of $[0,T)$ to approximate a vector-valued function in $L_{p}([0,T),X)$.
This, in turn, is an intermediate tool for constructing the optimal time-space partition.

In the rest of this section we consider the following framework. We let $X$ denote a Banach space,  $r\in\nat$ denotes the polynomial \emph{order} with respect to time, and for an interval $I$, recall the definition of $\V_{I,X}^r$ from~\eqref{Vr}:
\[
%\mathbb{V}_{I,X}:=
\mathbb{V}^r_{I,X}:=\Big\{ P(t)=\sum_{j=0}^{r-1}a_j t^j, \ a_j\in X, \ t\in I\Big\} \subset L_p(I,X),
\]
i.e., the tensor product space $\Pi^r \otimes X$ on the time slice $I\times \Omega$.
For a partition $\calT = \{ 0=t_0 < t_1 < \dots < t_N = T \}$ of the time interval $[0,T)$, we consider the following corresponding (abstract) finite element space:
\[
\V_{\calT,X}^r = \{ P \in L_p([0,T),X) :\ P_{|I} \in \V_{I,X}^r,\ I \in \calT \}.
\]

Recall the definition of the best approximation error $E_r(f,I)_p$ associated with an interval $I\subset [0,T)$, i.e.,
\begin{equation}
E_r(f,I)_p = \inf_{P_I\in \mathbb{V}_{I,X}^r}\|f-P_I\|_{L_p(I, X)},
\end{equation}
so that
\[
\inf_{g \in \V_{\calT,X}^r} \| f - g \|_{L_p([0,T),X)}
= \left(\sum_{I \in \calT} E_r(f,I)_p^p \right)^{1/p}.
\]

An algorithm approximating the solution with a parameter $\delta >0$ reads as follows: \\

\begin{algorithm}[h]
\caption{\emph{Greedy} algorithm}\label{a:greedy}
\begin{algorithmic}[1]
\Function{Greedy}{$f$,$\delta$}
\State{Let $\mathcal{T}_0 = \{0<T\} = \{ [0,T) \}$.}
%\State{Let  $\M_0 = \{ I \in \calT_0:  E_r(f,I)_p > \delta \}$.}
\State{$k = 0$}
\While{$\M_k :=  \{ I \in \calT_k :  E_r(f,I)_p > \delta \} \ne \emptyset$}
\State{Let $\mathcal{T}_{k+1} = \textsc{Refine}(\mathcal{T}_k,\M_k)$}
\State{$k \leftarrow k+1$}
% \State{Let $\M_k = \{ I \in \calT_k :  E_r(f,I)_p > \delta \}$.}
\EndWhile
\EndFunction
\end{algorithmic}

\end{algorithm}

\subsection{Semi-discretization in time}

Concerning the error when approximating a vector-valued function with piecewise polynomials with respect to time, we have the following result.

\begin{thm}[{\bf Time discretization}]\label{time-poly}
Let $X$ be a separable Banach space, let $s>0$, $0<p,q\leq \infty$, and  $\left(\frac 1q-\frac 1p\right)_+ \le s < r$, with $r \in \N$.
Then, if $f\in B^s_{q,q}([0,T),X)$ and $\varepsilon > 0$, there exists $\delta > 0$ such that \textsc{Greedy($f$,$\delta$)} terminates in finitely many steps and the generated partition $\mathcal{T}$ satisfies
\beq\label{boundonT}
\# \mathcal{T}\le c_1 \, %{\delta^{-\frac 1{s+1/2}}}
\varepsilon^{-1/s}, %|f|^{\frac 1s}_{B^s_{q,q}([0,T),X)},
\eeq
where the constant $c_1>0$ depends on $p$, $q$, and $s$ but not on $f$.
Moreover, there exists $P \in \V_{\calT,X}^r$ satisfying %for each $I \in \calT$ there exists $g_I \in \V_{I,X}^r$ such that, if $g(t) = \sum_{I \in \calT} \chi_I(t) g_I(t)$, with $g_I$ equal to the $L_2(I,X)$-projection of $u_{|I}$ into $\V_{I,X}^r$, then $g \in \V_{\calT,X}^r$ and
\beq\label{boundonerror}
\|f-P\|_{L_p([0,T), X)} \le c_2 \, \varepsilon \, |f|_{B^s_{q,q}([0,T),X)}
\le c_3 \, {(\# \mathcal{T})^{-s}}|f|_{B^s_{q,q}([0,T),X)},
\eeq
with $c_2,c_3>0$ depending on $p$, $q$, and $s$ but not on $f$.
\end{thm}

\begin{proof}
Let $\varepsilon>0$ be given and let $\delta = \varepsilon^{\frac{s+1/p}s} |f|_{B^s_{q,q}([0,T),X)}$.
Using Whitney's estimate \eqref{whitney} we see that the error $ E_r(f,I)_p$ associated with an interval $I$ satisfies
\begin{equation}\label{error_est}
 E_r(f,I)_p = \inf_{P_I \in \V_{I,X}^r} \| f - P_I \|_{L_p(I,X)} \lesssim |I|^{s+\frac 1p-\frac 1q}|f|_{B^s_{q,q}(I,X)}.
\end{equation}
Since $s+\frac {1}{p}-\frac 1q>0$ the right-hand side goes to zero as $|I|$ goes to zero, which shows that the Greedy algorithm terminates in a finite number of steps $K$.

We now bound the number of elements of $\mathcal{T}:= \calT_K$ as follows. Initially, $\mathcal{T}_0=\{[0,T)\}$, therefore, $\# \mathcal{T}_0=1$. In each iteration of the {\bf while}-loop, $\#\mathcal{M}_k$ elements are \emph{marked} for refinement. If $\overline{\mathcal{M}}=\bigcup_{k=0}^{K-1}\mathcal{M}_k$ is the union of all marked elements in a certain step of the algorithm, then, due to~\eqref{eq:complexity}, the resulting final partition $\mathcal{T}$ satisfies $\# \mathcal{T}\lesssim 1 + \# \overline{\mathcal{M}} \lesssim \# \overline{\mathcal{M}}$.
We see that estimating $\# \mathcal{T}$ is comparable with estimating $\# \overline{\mathcal{M}}$. In order to count the number of elements in $\overline{\mathcal{M}}$ observe that $\#\overline{\mathcal{M}}=\sum_{k=0}^{\infty}\#\mathcal{M}^k$,
with
\[
\calM_k=\left\{I\in \overline{\calM}: |I| = T/2^k\right\}
\ \text{ if }\ 0 \le k \le K-1
\qquad \text{and}\qquad  \calM_k = \emptyset
\ \text{ if }\ k \ge K.
\]
On the one hand, since our time interval $[0,T)$ is finite, we obtain the upper bound
\[
\#\calM_k \le  2^k, \qquad k\in \nat_0.
\]
On the other hand, if $I\in \calM_k$ from steps 4 and 6 of the Greedy algorithm and formula \eqref{error_est},  we have
\[
\delta< E_r(f,I)_p \lesssim \left(\frac{1}{2^k}\right)^{s+\frac 1p-\frac 1q}|f|_{B^s_{q,q}(I,X)},
\quad\text{so that}\quad
\delta^q\lesssim \left(\frac{1}{2^k}\right)^{sq+\frac qp-1}|f|^q_{B^s_{q,q}(I,X)}.
\]
This implies
\[
\delta^q \# \calM_k =\sum_{I\in \calM_k}\delta^q
\lesssim \left(\frac{1}{2^k}\right)^{sq+\frac qp-1}\sum_{I\in \calM_k}|f|^q_{B^s_{q,q}(I,X)}
\le \left(\frac{1}{2^k}\right)^{sq+\frac qp-1}|f|^q_{B^s_{q,q}([0,T),X)},
\]
i.e.,
\[
\# \calM_k \lesssim \min \left\{2^k, \frac{1}{\delta^q}\left(\frac{1}{2^k}\right)^{sq+\frac qp-1}|f|^q_{B^s_{q,q}([0,T),X)}\right\}.
\]
The first term corresponds to an increasing geometric series, the second to a decreasing one. Setting  $k_0:=\min \left\{k\in \nat_0: \ \frac{1}{\delta^q} \left(\frac{1}{2^k}\right)^{sq+\frac qp-1}|f|^q_{B^s_{q,q}([0,T),X)}<2^k\right\}$  we obtain
\begin{align}
\# \overline{\calM}=\sum_{k=0}^\infty \#\calM_k
&\leq \sum_{k=0}^{k_0-1}2^k+\sum_{k=k_0}^{\infty}\frac{1}{\delta^q} \left(\frac{1}{2^k}\right)^{sq+\frac qp-1}|f|^q_{B^s_{q,q}([0,T),X)}
\notag
\\
& \lesssim 2^{k_0}+\frac{1}{\delta^q} \left(\frac{1}{2^{k_0}}\right)^{sq+\frac qp-1}|f|^q_{B^s_{q,q}([0,T),X)}
\lesssim 2^{k_0}.
\label{one}
\end{align}
In order to estimate $2^{k_0}$ we observe that
\[
2^{k_0-1}\le \frac{1}{\delta^q} \left(\frac{1}{2^{k_0}}\right)^{sq+\frac qp-1}|f|^q_{B^s_{q,q}([0,T),X)}<2^{k_0},
\]
\[
2^{k_0(s+\frac 1p)-\frac 1q}\le \frac{1}{\delta}|f|_{B^s_{q,q}([0,T),X)}<2^{k_0(s+\frac 1p)}.
\]
We see that
\beq\label{two}
2^{k_0}\le \left(\frac{1}{\delta}\right)^{\frac1{s+1/p}}|f|^{\frac1{s+1/p}}_{B^s_{q,q}([0,T),X)},
\eeq
therefore, from~\eqref{one} and~\eqref{two} we get 
$$
\# \calT \lesssim \# \overline{\calM}\lesssim \left(\frac{1}{\delta}\right)^{\frac1{s+1/p}}|f|^{\frac1{s+1/p}}_{B^s_{q,q}([0,T),X)}, \quad \text{i.e.}, \quad \delta \lesssim (\# \calT)^{-(s+\frac1p)}|f|_{B^s_{q,q}([0,T),X)},
$$
and~\eqref{boundonT} follows after recalling that $\delta = \varepsilon^{\frac{s+1/p}s} |f|_{B^s_{q,q}([0,T),X)}$.

Finally, for each $I \in \calT$ we let $P_I \in \V_{I,X}^r$ satisfy $\| f - P_I\|_{L_p(I,X)} \le {2 E_r(f,I)_p}$ and let $P(t) = \sum_{I \in \calT} \chi_I(t) P_I(t)$, $t \in [0,T)$. Hence,
\begin{align*}
P \in \V_{\calT,X}^r \quad\text{and}\quad
\| f - P \|_{L_p([0,T),X)}^p &\lesssim \delta^p \#\calT
\lesssim \#\calT^{-ps} |f|_{B^s_{q,q}([0,T),X)}^p,
\end{align*}
and~\eqref{boundonerror} follows.
\end{proof}

\section{Discretization in time and space}\label{sec:timespace}

We now consider the error when approximating a function with piecewise polynomials with respect to time and space. In this article, we deal with the approximation in $L_2([0,T)\times\Omega) = L_2([0,T),X)$, where hereafter we let $X = L_2(\Omega)$.
{We restrict ourselves to this Hilbertian case in order to avoid additional technical difficulties and leave the study of more general quasi-norms, e.g. $p\neq 2$ and $X \neq L_2(\Omega)$, to a forthcoming article.}

\subsection{{Time marching fully discrete adaptivity}}

Recall that the type of discretizations that we consider are those consisting of a partition $\{ 0=t_0<t_1<\dots<t_N=T \}$ of the time interval and a sequence of partitions $\mathcal{T}_1,\dots, \mathcal{T}_N \in \TT$ of the space domain $\Omega$, where $\mathcal{T}_i$ corresponds to the subinterval $[t_{i-1},t_i)$, $i=1,\dots N$.
The time-space partition is then given by
\[
\mathcal{P}=\left(\{0=t_0<t_1<\ldots<t_N=T\}, \{\mathcal{T}_1,\ldots, \mathcal{T}_N\}\right),
\quad \text{with }\# \mathcal{P}=\sum_{i=1}^N \#\mathcal{T}_i,
\]

Given $r_1,r_2 \in\N$, the finite element space $\overline{\mathbb{V}}_{\mathcal{P}}^{r_1,r_2}$ subject to such a partition $\mathcal{P}$ is defined as
\begin{align*}
% \overline{\mathbb{V}}_{\mathcal{P}}^{r_1,r_2} &:=\{v: \Omega\times [0,T)\rightarrow \real: \ & v(\cdot, t)\in \mathbb{V}_{\mathcal{T}_i}\ \text{if}\ t\in [t_{i-1},t_i) \ \text{ and } \\
%  &v(x,\cdot)\big|_{[t_{i-1},t_i)}\in \mathcal{Q}^l\text{ for all }x\in \Omega\}\\
\overline{\mathbb{V}}_{\mathcal{P}}^{r_1,r_2} &:=\{G:  [0,T)\times\Omega\rightarrow \real:  G_{\big|[t_{i-1},t_i)\times\Omega} \in \Pi^{r_1} \otimes \mathbb{V}_{\mathcal{T}_i}^{r_2},\text{ for all }i=1,2,\dots,N \},
\end{align*}
i.e., $G \in \overline{\mathbb{V}}_{\mathcal{P}}^{r_1,r_2}$ if and only if $G(t,\cdot)\in \mathbb{V}_{\mathcal{T}_i}^{r_2}$ for all $t\in [t_{i-1},t_i) $ and $G(\cdot,x)\big|_{[t_{i-1},t_i)}\in \Pi^{r_1}$ for all $x\in \Omega$ and all $i=1,2,\dots,N$.

%%%%%%%%%%%%%%%%%%%%
\begin{comment}
\begin{equation}
%\varepsilon:=
\inf_{g\in \overline{\mathbb{V}}_{\mathcal{P}}^{r_1,r_2}}\|f-g|L_2([0,T),X)\|,
\end{equation}
%The approximation error is given by
%\[
%\overline{\sigma}_N(f)=\inf_{\# \mathcal{P}\leq N}\inf_{g\in \overline{\mathbb{V}}_{\mathcal{P}}^{r_1,r_2}}\|f-g|L_p([0,T), X)\|,
%\]
where hereafter $X = L_2(\Omega)$.
\end{comment}

In order to construct an optimal approximate solution with tolerance $\varepsilon>0$ we use the one-dimensional Greedy algorithm as described on page \pageref{a:greedy} for the (adaptive) discretization in time and an $n$-dimensional Greedy algorithm for (adaptive) discretizations in space. This allows us to use the results from Theorems \ref{time-poly} and \ref{space-poly}, respectively. In particular, we obtain the following result.

%\todo[inline]{We can only prove this result for $q_2 \ge 1$. We will have to explain why!!}

\begin{thm}[{\bf Approximation with fully discrete functions}]\label{space-time-poly}
Let $0 < s_i < r_i$, $ i=1,2$, $0 < q_1 \le \infty$, $1\le q_2 \le \infty$ with $s_1 > \big(\frac1{q_1}-\frac12\big)_+$ and 
$s_2 > n (\frac1{q_2}-\frac12\big)_+$.
Let $f\in B^{s_1}_{q_1,q_1}([0,T),X) \cap L_2([0,T),B_{q_2,q_2}^{s_2}(\Omega))$, with $X = L_2(\Omega)$.
Then, for each $\varepsilon > 0$ there exists a time-space partition $\calP$ that satisfies
\[
\# \mathcal{P}\le c_1 \varepsilon^{-\big(\frac1{s_1}+\frac {n}{s_2}\big)}
\]
and a function $F \in \overline{\V}_{\mathcal{P}}^{r_1,r_2}$ such that
\begin{align*}
\| f - F \|_{ L_2([0,T),X)}
\le c_2 \varepsilon
|\!|\!| f |\!|\!|
\le c_3 (\#\mathcal{P})^{-\frac{1}{\frac1{s_1}+\frac {n}{s_2}}} |\!|\!| f |\!|\!|,
\end{align*}
where $|\!|\!| f |\!|\!| = | f |_{B^{s_1}_{q_1,q_1}([0,T),X)} + \| f \|_{L_2([0,T),B_{q_2,q_2}^{s_2}(\Omega))}$ and the positive constants $c_1,c_2,c_3$ depend on $q_1,q_2$, and $s_1,s_2$ but not on $f$.
% In particular,
% \[
% \inf_{g\in \overline{\mathbb{V}}_{\mathcal{P}}^{r_1,r_2}}\|f-g|L_2([0,T), X)\|\lesssim {(\# \mathcal{P})^{-\frac{1}{\frac 1{s_1}+\frac d{s_2}}}}|f|_{B^s_{q,q}([0,T),B^{\alpha+\sigma}_{\tau,\tau}(\Omega))}.
% \]
\end{thm}

\begin{rem}
Here (with a little abuse)  we use the notation 
\[
L_2(I,B_{q,q}^s(\Omega)) = \left\{ f : I \to B_{q,q}^s(\Omega) : \| f \|_{ L_2(I,B_{q,q}^s(\Omega))}  < \infty \right\}
\]
with $ \| f \|_{ L_2(I,B_{q,q}^s(\Omega))} := \left(\int_I \| f(t) \|_{B_{q,q}^s(\Omega)}^2 \ud t\right)^{1/2} $. 

 The restriction $q_2\geq 1$ in Theorem \ref{space-time-poly} can probably be removed  and replaced by $q_2>0$. 
 It appears here %is 
 due to the fact that we require in the proof below a uniform bound of the approximants on a subinterval $I=[t_{i-1},t_i)$, which is established in Lemma \ref{L:uniformbound}. %However, the 
 Our current proof of Lemma  \ref{L:uniformbound} uses Minkowski's inequality which only works if $q_2\geq 1$. So far we were not able to find an  appropriate modification for $q_2<1$. 
\end{rem}

\begin{proof}[Proof of Theorem \ref{space-time-poly}]
Given $f \in B^{s_1}_{q_1,q_1}([0,T),X) \cap L_2([0,T),B_{q_2,q_2}^{s_2}(\Omega))$ and $\varepsilon > 0$, the approximant  $F\in \overline{\mathbb{V}}_{\mathcal{P}}^{r_1,r_2}$ is constructed in two steps as follows.

We first use a one-dimensional Greedy algorithm and apply the results from Theorem~\ref{time-poly}.
This gives a partition of the time interval
$0=t_0<t_1<\ldots < t_N=T$ and an approximant $G=\sum_{i=1}^N{\chi}_{[t_{i-1},t_i)}G_i \in \V_{\{0<t_1<\dots<T\},X}^{r_1}$ with $G_i$ the $L_2([t_{i-1},t_i),X)$ projection of $f_{|[t_{i-1},t_i)}$ into $\V_{[t_{i-1},t_i),X}^{r_1}$.
This partition and approximant satisfy
\[
N \lesssim \varepsilon^{-1/s_1} \quad\text{and}\quad
\|f-G\|_{L_2([0,T), X)}
\lesssim \varepsilon | f |_{ B^{s_1}_{q_1,q_1}([0,T),X)}.
\]
Also, if $\{\onb{i}{j}\}_{j=1}^{r_1}$ is an orthonormal basis of $\mathbb{V}^{r_1}_{[t_{i-1},t_i),\R}$ then
\[
G_i(t) = \sum_{j=1}^{r_1} G_i^j \, \onb{i}{j}(t),
\quad\text{with}\quad
G_i^j = \int_I f(t) \onb{i}{j}(t) \, \ud t,
\]
noting that the last integral is a Bochner integral in $X = L_2(\Omega)$.

We now observe that due to Lemma~\ref{L:uniformbound} below we have $G_i^j  \in B^{s_2}_{q_2,q_2}(\Omega)$   and 
\begin{equation}\label{uniformbound}
\| G_i \|_{ L_2([t_{i-1},t_i),B^{s_2}_{q_2,q_2}(\Omega) )} \lesssim \| f \|_{ L_2([t_{i-1},t_i),B^{s_2}_{q_2,q_2}(\Omega) )}.
\end{equation}

The second step consists in approximating each function $G_i^j  \in B^{s_2}_{q_2,q_2}(\Omega)$   using the space-adaptive Greedy algorithm. Resorting to Corollary~\ref{cor:spaceapprox} we find a mesh $\calT_i^j \in \TT(\calT_0)$ and a finite element function $F_i^j \in \V_{\calT_i^j}^{r_2}$ with
\[
\#\calT_i^j \lesssim \varepsilon^{-\frac{n}{s_2}}
\quad\text{and}\quad
\| G_i^j - F_i^j \|_X \lesssim \varepsilon | G_i^j |_{B^{s_2}_{q_2,q_2}(\Omega)}.
\]
Therefore, after defining $\calT_i = \oplus_{j=1}^{r_1} \calT_i^j$ (the overlay of the meshes~\cite{CKNS08}), we have that $F_i(t) := \sum_{j=1}^{r_1} \onb{i}{j}(t) F_i^j \in \V_{[t_{i-1},t_i),\V_{\calT_i}^{r_2}}^{r_1}$ satisfies
\begin{align*}
\#\calT_i \le \sum_{j=1}^{r_1} \#\calT_i^j &\lesssim \varepsilon^{-\frac{n}{s_2}}
\quad\text{\cite[Lem.~3.7]{CKNS08} and}\quad \\
\| F_i - G_i \|_{ L_2([t_{i-1},t_i),X )}
&
\lesssim \varepsilon
\| G_i \|_{ L_2([t_{i-1},t_i),B^{s_2}_{q_2,q_2}(\Omega)) }
\lesssim
\varepsilon
\| f \|_{ L_2([t_{i-1},t_i),B^{s_2}_{q_2,q_2}(\Omega)) },
\end{align*}
due to~\eqref{uniformbound}.

Finally, we let $\calP = \{ \{ 0=t_0 < t_1 < \dots < t_N = T\}, \{ \calT_1, \calT_2, \dots , \calT_N \} \}$ and define $F = \sum_{i=1}^N \chi_{[t_{i-1},t_i)} F_i \in \overline{\V}_\calP^{r_1,r_2}$,
whence by the triangle inequality
\begin{align*}
\|f-F\|_{L_2([0,T),X)}
& \leq \|f-G\|_{L_2([0,T),X)} + \|G-F\|_{L_2([0,T),X)}
\\
& \lesssim \varepsilon | f |_{ B^{s_1}_{q_1,q_1}([0,T),X)}
+  \bigg( \sum_{i=1}^N \|  G_i - F_i \|_{ L_2([t_{i-1},t_i),X)}^2 \bigg)^{1/2}
\\
% & \textcolor{red}{\lesssim \varepsilon | f |_{ B^{s_1}_{q_1,q_1}([0,T),X)}
% +  \bigg( \sum_{i=1}^N \|  g_i - f_i \|_{ L_2([t_{i-1},t_i),B^{s_2}_{q_2,q_2}(\Omega) ))}^2 \bigg)^{1/2}}
% \\
%& \lesssim \varepsilon | f |_{ B^{s_1}_{q_1,q_1}([0,T),X)}
%+ \varepsilon \bigg( \sum_{i=1}^N \| \cosc G_i \scco  \|_{ L_2([t_{i-1},t_i),B^{s_2}_{q_2,q_2}(\Omega) ) }^2 \bigg)^{1/2}
%\\
&\lesssim \varepsilon \Big( | f |_{ B^{s_1}_{q_1,q_1}([0,T),X)}
+
\| f \|_{ L_2([t_{i-1},t_i),B^{s_2}_{q_2,q_2}(\Omega))} \Big)
\end{align*}
and
\begin{align*}
\# \mathcal{P}
=\sum_{i=1}^N \# \mathcal{T}_i
\lesssim  N \, \varepsilon^{-\frac n{s_2}}
\lesssim \varepsilon^{-\frac{1}{s_1}} \, \varepsilon^{-\frac n{s_2}}
 =   \varepsilon^{-(\frac 1{s_1}+\frac n{s_2})} .
\end{align*}
The assertion of the theorem thus follows.
\end{proof}

In Theorem \ref{space-time-poly}, formula \eqref{uniformbound}, we required a uniform bound of the approximants $G_i$ on a subinterval $I=[t_{i-1},t_i)$, which is provided by the following lemma.

\begin{lem}\label{L:uniformbound}
Given a finite interval $I$, let $r=r_1$, $s=s_2$, and $q=q_2$ satisfy the assumptions from Theorem~\ref{space-time-poly}  and  assume $f \in L_2(I,B_{q,q}^s(\Omega))$.    If $G \in \V_{I,X}^r$ is the $L_2(I,X)$ projection of $f \in L_2(I,B_{q,q}^s(\Omega))$, then
\[
\| G \|_{ L_2(I,B^{s}_{q,q}(\Omega) )} \lesssim \| f \|_{ L_2(I,B^{s}_{q,q}(\Omega) )}.
\]
\end{lem}

\begin{proof}
If $\{\onb{}{j}\}_{j=1}^{r}$ is an orthonormal basis of $\V_{I,\R}^r$ then
\[
G(t) = \sum_{j=1}^{r} G^j \onb{}{j}(t)
\quad\text{with}\quad
G^j = \int_I f(t) \onb{}{j}(t) \ud t,
\]
i.e., 
\[
G(t)(x) = \sum_{j=1}^{r} G^j(x) \onb{}{j}(t)
\quad\text{with}\quad
G^j(x) = \int_I f(t,x) \onb{}{j}(t) \ud t,
\]
for almost every $x \in \Omega$.
% where the last integral is a Bochner integral for functions valued in $X = L_2(\Omega)$ \textcolor{red}{(shouldn't it be $X=B^{s}_{q,q}(\Omega)$?)}, but also
% \[
% G^j(x) = \int_I f(t,x) \onb{}{j}(t) \ud t 
% \quad\text{for almost every $x \in \Omega$}.
% \]
Notice first that
    \begin{align*}
    \| G \|_{L_2(I,B^s_{q,q}(\Omega))}^2
        &= \int_I \| G(t) \|_{L_q(\Omega)}^2 + | G |_{B_{q,q}^s(\Omega)}^2 \ud t
        \nonumber\\
        &=  \int_I \Big\| \sum_{j=1}^r G^j \onb{}{j}(t) \Big\|_{L_q(\Omega)}^2 + \Big|  \sum_{j=1}^r G^j \onb{}{j}(t) \Big|_{B_{q,q}^s(\Omega)}^2 \ud t
        \nonumber\\
        &\lesssim \sum_{j=1}^r \| G^j \|_{L_q(\Omega)}^2+  |G^j|_{B_{q,q}^s(\Omega)}^2 ,
\end{align*}
so that
\begin{equation}\label{eq:G_bound2}
   \| G \|_{L_2(I,B^s_{q,q}(\Omega))} \lesssim \sum_{j=1}^r \| G^j \|_{L_q(\Omega)}+  |G^j|_{B_{q,q}^s(\Omega)}.
   \end{equation}

We now bound $\| G^j \|_{L_q(\Omega)}$ and $ |G^j|_{B_{q,q}^s(\Omega)} $ and focus on  the case $1\le q < \infty$, noting that the case $q=\infty$ is analogous. 
%, considering separately the cases $1\le q \le \infty$ and $0<q<1$.
%\medskip
%\noindent$\bullet$ \ \textbf{We first focus on $\| G^j \|_{L_q(\Omega)}$.}
%\noindent$\mathbf{-}$ \ 
Since $q \ge 1$, by Minkowski's inequality, for any $j=1,2,\dots, r$ %(if $q < \infty$) 
we have 
\begin{align*}
\| G^j\|_{ L_q(\Omega)}
    & = \left(\int_\Omega \left|\int_I f(x,t)W^j(t)\ud t \right|^q \ud x\right)^{1/q}\nonumber\\
    %(\text{Minkowski })
    &\le \int_I \left(\int_\Omega  \left|f(x,t)W^j(t) \right|^q \ud x \right)^{1/q}\ud t \nonumber\\
    &= \int_I \left|W^j(t)\right| \left\|f(\cdot,t)\right\|_{L_q(\Omega)} \ud t \nonumber\\
    &\le \left\| \onb{}{j}(t) \right\|_{L_2(I)}  \left\| \left\|f(\cdot,t)\right\|_{L_q(\Omega)}   \right\|_{L_2(I)}
    %( \text{H\"older with 2 })
    \end{align*}
    so that
    \begin{equation}\label{eq:G_bound0}
\| G^j\|_{ L_q(\Omega)} \lesssim  \|f\|_{L_2(I,L_q(\Omega))}, \quad j=1,2,\dots, r. 
    \end{equation}

We now  deal with    $\left|G^j\right|_{B^s_{q,q}(\Omega)} $.  
%and prove the case $1\le q < \infty$, noting that the case $q=\infty$ is analogous.
Observe that for any $j$ we have
\begin{align}
\left|G^j\right|_{B^s_{q,q}(\Omega)}
&\lesssim \left(\int_0^1 \left[u^{-s}w_r(G^j,I,u)_q\right]^q \frac{\ud u}{u}\right)^{1/q} 
= \left(\int_0^1 u^{-sq}w_r(G^j,I,u)_q^q \frac{\ud u}{u}\right)^{1/q} \nonumber
\\
&= \left(\int_0^1 u^{-sq} \frac{1}{(2u)^{n}} \int_{|h|\le u} \left\|\Delta_h^rG^j\right\|^q_{L_q(\Omega_{ rh })} \ud h \frac{\ud u}{u}\right)^{1/q}\nonumber
\\
&=
\left(\int_0^1 u^{-sq}\frac{1}{(2u)^{n}} \int_{|h|\le u}
\int_{\Omega_{ rh }}\left|\int_I \Delta_h^rf(t,x) \onb{}{j}(t) \ud  t\right|^q\ud x
\ud h  \frac{\ud u}{u}\right)^{1/q}\nonumber
\\
&=
\left(\int_0^1 \int_{|h|\le u}  \int_{\Omega_{ rh }}
u^{-sq}\frac{1}{(2u)^{n}} \left|\int_I \Delta_h^rf(t,x) \onb{}{j}(t) \ud  t\right|^q\ud x
\ud h  \frac{\ud u}{u}\right)^{1/q}. \nonumber
\end{align}

%\noindent$\mathbf{-}$ \ \textbf{If $1\le q < \infty$}, 
Again, by Minkowski's inequality
\begin{align*}
\left|G^j\right|_{B^s_{q,q}(\Omega)}
&\le
\int_I \left(\int_0^1 \int_{|h|\le u}  \int_{\Omega_{ rh }}
u^{-sq}\frac{1}{(2u)^{n}} \left| \Delta_h^rf(t,x) \onb{}{j}(t)\right|^q\ud x
\ud h  \frac{\ud u}{u}\right)^{1/q}  \ud  t \nonumber
\\
&=
\int_I  \left| \onb{}{j}(t) \right| \left(\int_0^1  u^{-sq} \frac{1}{(2u)^{n}} \int_{|h|\le u}
 \int_{\Omega_{rh }} \left|\Delta_h^rf(t,x)\right|^q \ud x
\ud h  \frac{\ud u}{u}\right)^{1/q} \ud t\nonumber
% \\
%&=
%\int_I \left|\onb{}{j}(t)\right| \left(\int_{I}\left[ u^{-s}\left\{\frac{1}{(2u)^{n}} \int_{|h|\le u}
% \int_{\Omega_h} \left|\Delta_h^rf(t,x)\right|^q  \ud x
%\ud h\right\}^{1/q} \right]^{q} \frac{\ud u}{u}\right)^{1/q} \ud t\nonumber
\\
&=
\int_I \left|\onb{}{j}(t)\right| \left|f(t)\right|_{B^s_{q,q}(\Omega)}   \ud t
%\nonumber\\&
\leq
\left\|\onb{}{j}\right\|_{L_2(I)}\left\|f(t)\right\|_{L_2(I,B^s_{q,q}(\Omega))}
%\nonumber\\
%&= \left\|f(t)\right\|_{L_2(I,B^s_{q,q}(\Omega))}
\end{align*}
whence
\begin{equation}
\label{eq:G_bound1}
 \left|G^j\right|_{B^s_{q,q}(\Omega)}
 \lesssim
 \left\|f\right\|_{L_2(I,B^s_{q,q}(\Omega))}. 
\end{equation}

Therefore from \eqref{eq:G_bound0} \eqref{eq:G_bound1} and \eqref{eq:G_bound2} we get
\[
\| G \|_{ L_2(I,B^{s}_{q,q}(\Omega) )} \lesssim \| f \|_{ L_2(I,B^{s}_{q,q}(\Omega) )} 
\qquad\text{if \ $1 \le q < \infty$},
\]
and analogously for $q = \infty$.
\end{proof}

\begin{comment}
\pedro{
\begin{lem}\label{L:auxiliar}
Let $0 < q < 1$ and let $f \in L_{1+q}(I)$ for some interval $I$ of positive length $|I|$.
Then
\[
\| f \|_{L_1(I)}
\le q^2{|I|^{1-\frac1q}} \| f\|_{L_q(I)} 
+ (1-q^2) |I|^{\frac{q}{1+q}} \| f \|_{L_{1+q}(\Omega)}
\]
\end{lem}
}

\begin{proof}
\pedro{
We start applying Hölder's inequality with exponents $p = \frac1q$ and $p'=\frac1{1-q}$ to bound
\begin{align*}
\int_I |f|  &= \int_I |f|^{q^2} |f|^{1-q^2} 
\\
&\le \left(\int_I|f|^{\frac{q^2}q} \right)^{q} \left(\int_I |f|^{\frac{1-q^2}{1-q}}\right)^{1-q}
\\
&= \left(\int_I|f|^q\right)^{q} \left(\int_I |f|^{1+q}\right)^{1-q}
= \| f \|_{L_q(I)}^{q^2} \| f \|_{L_{1+q}(I)}^{1-q^2}.
\end{align*}
}
\pedro{
We now apply Young's inequality $ab \le \frac{a^p}{p} + \frac{b^{p'}}{p'}$ with $p = \frac1{q^2}$ and $p'= \frac1{1-q^2}$ and $a = |I|^{q^2-q}\| f \|_{L_q(I)}^{q^2}$ and $b = |I|^{q-q^2} \| f \|_{L_{1+q}(I)}^{1-q^2}$ to obtain
\[
\int_I |f| \le 
q^2\left(|I|^{q^2-q}\| f \|_{L_q(I)}^{q^2}\right)^{\frac{1}{q^2}}
+
(1-q^2) \left( |I|^{q-q^2} \| f \|_{L_{1+q}(I)}^{1-q^2} \right)^{\frac{1}{1-q^2}},
\]
which is the desired assertion.
}
\end{proof}

\end{comment}

If we use the same polynomial degree in space and time in Theorem \ref{space-time-poly}  the result reads  as follows.

\begin{cor}[{\bf Fully discrete with same polynomial degree}]\label{space-time-poly-2}
Let $1\le q \le \infty$ and $n\Big(\frac1q-\frac12\Big)_+ < s < r \in \N$. If $f\in B^{s}_{q,q}([0,T),X) \cap L_2([0,T),B_{q,q}^{s}(\Omega))$ with $X = L_2(\Omega)$,
then for each $\varepsilon > 0$ there exists a time-space partition $\calP$ that satisfies
\[
\# \mathcal{P}\le c_1 \varepsilon^{-\frac{n+1}{s}}
\]
and a function $F \in \overline{\V}_{\mathcal{P}}^{r,r}$ such that
\begin{align*}
\| f - F \|_{ L_2([0,T),X)}
\le c_2 \varepsilon
|\!|\!| f |\!|\!|
\le c_3 (\#\mathcal{P})^{-\frac{s}{n+1}} |\!|\!| f |\!|\!|,
\end{align*}
where $|\!|\!| f |\!|\!| = | f |_{B^{s}_{q,q}([0,T),X)} + \| f \|_{L_2([0,T),B_{q,q}^{s}(\Omega))}$ and the positive constants $c_1,c_2,c_3$ depend on $q$ and $s$ but not on $f$.
% 
% Under the assumptions of Theorem \ref{space-time-poly}, for $1\leq  q\le \infty$ and $r > s_1=s_2 = s > n\Big(\frac1q-\frac12\Big)_+$ with $r \in \N$ we have
% \[
% \inf_{G\in \overline{\mathbb{V}}_{\mathcal{P}}^{r,r}}\|f-G\|_{L_2([0,T), L_2(\Omega))}
% \lesssim {(\# \mathcal{P})^{-\frac{s}{n+1}}}
% \Big(| f |_{B^{s}_{q,q}([0,T),X)} + \| f \|_{L_2([0,T),B_{q,q}^{s}(\Omega)) }\Big).
% \]
\end{cor}

%\section{Discussion and further comments}\label{sec:discussion}

\subsection{Comparison with space-time finite elements}

If we were to use space-time finite elements of order $r$ in $\real^{n+1}$, in order to obtain the same rate $(\#\mathcal{P})^{-\frac{s}{n+1}}$ as that indicated in Corollary~\ref{space-time-poly-2}, Corollary~\ref{cor:spaceapprox} tells us that the function $f$ should belong to $B^s_{q,q}([0,T)\times\Omega)$
with $0<s  <   r$ and  $0<\frac1q < \frac{s}{n+1}+\frac12$.
This raises the following question:
\begin{quote}
What is the relation between the spaces
\[
B^s_{q,q}([0,T)\times\Omega)\quad\text{and}\quad
B^{s}_{q_1,q_1}([0,T),L_2(\Omega)) \cap L_2([0,T),B_{q_2,q_2}^{s}(\Omega))
\]
for the respective ranges of the parameters $q_1,q_2,$ and $q$?
\end{quote}

The following proposition provides a first attempt to  give an answer to this question.

\begin{prop}
Let $0<s<r$ and $0<q_1,q_2,q < \infty$,  where we additionally require that
\beq \label{range_q}
\frac1q < \frac{s}{n+1}+\frac12, \qquad \frac{1}{q_1} < {s}+\frac12, \qquad \text{and}\qquad  \frac{1}{q_2} < \frac{s}{n}+\frac12.
\eeq
Then we have
\beq
\bigcup_{q_1,q_2}B^{s}_{q_1,q_1}([0,T),L_2(\Omega))\cap L_2([0,T),B_{q_2,q_2}^{s}(\Omega))\not\subset \bigcup_{q}B^s_{q,q}([0,T)\times \Omega),
\eeq
where the union is taken over all $q,q_1,q_2$ according to \eqref{range_q}.
\end{prop}

\begin{proof}
We show that we can find functions belonging to $ \bigcup_{q_1,q_2}B^{s}_{q_1,q_1}([0,T),L_2(\Omega))\cap L_2([0,T),B_{q_2,q_2}^{s}(\Omega))$ which are not in $\bigcup_{q}B^s_{q,q}([0,T)\times \Omega)$.
%We start with
%\[
%\bigcup_{q_1}B^{s}_{q_1,q_1}([0,T),L_2(\Omega))\not \subset  \bigcup_{q}B^s_{q,q}([0,T)\times \Omega).
%\]
For this let us choose $q_1$ such that
$$
\frac 1q< \frac{s}{n+1}+\frac 12<\frac{1}{q_1}< s + \frac 12
$$
and consider a function $f$ which is constant with respect to the space variable $x$ and belongs to $B^s_{q_1,q_1}([0,T))$. Clearly, by our assumptions this function is also in $L_2([0,T))$. Moreover,  by or choice of $q_1$ we see from \cite[Cor.~3.7]{HS13} that
\[
B^s_{q_1,q_1}([0,T),L_2(\Omega))\not\hookrightarrow B^{s}_{q,q}([0,T)\times\Omega),
\]
since  $q_1<q$, which proves the claim. Alternatively,
%we can  show that also
%\[
%\bigcup_{q_2}L_2([0,T),B^{s}_{q_2,q_2}(\Omega))\not \subset  \bigcup_{q}B^s_{q,q}([0,T)\times \Omega). %
%\]
we could choose  $q_2$ such that
$$
\frac 1q< \frac{s}{n+1}+\frac 12<\frac{1}{q_2}< \frac sn + \frac 12
$$
and consider a function $f$ which is constant with respect to the time variable $t$ and belongs to $B^s_{q_2,q_2}(\Omega)$. Then, clearly this function also belongs to $L_2(\Omega)$.  By our choice of $q_2$ it again follows  from \cite[Cor.~3.7]{HS13} that
\[
B^s_{q_2,q_2}(\Omega)\not\hookrightarrow B^{s}_{q,q}(\Omega),
\]
since  $q_2<q$. This completes the proof.
\end{proof}

\begin{rem} We believe that for the above spaces under consideration we actually have the following inclusion
\[
\bigcup_{q}B^s_{q,q}([0,T)\times \Omega)\subsetneq \bigcup_{q_1,q_2}B^{s}_{q_1,q_1}([0,T),L_2(\Omega))\cap L_2([0,T),B_{q_2,q_2}^{s}(\Omega)),
\]
where $q, q_1$, and $q_2$ are chosen according to \eqref{range_q}. This can be interpreted in the sense that the respective solution spaces for time-stepping algorithms yielding the approximation class $\mathbb{A}_{\frac{s}{n+1}}(L_2([0,T)\times \Omega))$ are in fact larger than the corresponding solution spaces for space-time finite elements. \\
In this context the fact that $B^s_{q,q}([0,T)\times \Omega)\subset L_2([0,T), B^s_{q_2,q_2}(\Omega))$ should be easier to handle. However, these matters are quite technical and, therefore,  this interesting question will be tackled in  a future paper.
\end{rem}

\bibliographystyle{alpha}

 \def\cprime{$'$}

\vfill

{
\bigskip
\smallskip
\vbox{\noindent  Marcelo Actis, Universidad Nacional del Litoral and CONICET, Departamento de Matem\'a{}tica, Faculdad de Ingenier\'i{}a Qu\'i{}mica, S3000AOM Santa Fe, Argentina\\
Phone: (+54) 342 457 1164 \\ 
E-mail: {\tt mactis@fiq.unl.edu.ar}
%Web: {\tt http://www.math.iisc.ac.in/postdocs.html}\\
}
}

{
\bigskip
\smallskip
\vbox{\noindent  Pedro Morin, Universidad Nacional del Litoral and CONICET, Departamento de Matem\'a{}tica, Faculdad de Ingenier\'i{}a Qu\'i{}mica, S3000AOM Santa Fe, Argentina\\
Phone: (+54) 342 457 1164 \\ 
E-mail: {\tt pmorin@fiq.unl.edu.ar}
%Web: {\tt http://www.math.iisc.ac.in/postdocs.html}\\
}
}

{
\bigskip
\smallskip
\vbox{\noindent  Cornelia Schneider,
Friedrich-Alexander-Universit\"at Erlangen-N\"urnberg,
Lehrstuhl AM3,
Cauerstr. 11,
91058 Erlangen, Germany\\
Phone: (+49) 9131 85-67207, Fax: (+49) 9131 85-67201  \\ 
E-mail: {\tt schneider@math.fau.de}
%Web: {\tt http://www.math.fau.de/$\sim$schneider/}\\}
}

\end{document}